\documentclass{amsart}
\usepackage[all]{xy}
\usepackage{amsmath,amsthm, graphicx}
\usepackage{color}
\usepackage{tikz-cd}
\usepackage[abs]{overpic}
 \usepackage{palatino}

\usepackage{hyperref}

\theoremstyle{plain}
\newtheorem{thm}{Theorem}[section]
\newtheorem{cor}[thm]{Corollary}

\newtheorem{prop}[thm]{Proposition}
\newtheorem{lem}[thm]{Lemma}

\theoremstyle{definition}
\newtheorem{remark}[thm]{Remark}
\newtheorem{example}[thm]{Example}

\newtheorem{question}[thm]{Question}

\newcommand{\comment}[1]{}

\newcommand{\Q}{\ensuremath{\mathbb{Q}}}
\newcommand{\R}{\ensuremath{\mathbb{R}}}
\newcommand{\Z}{\ensuremath{\mathbb{Z}}}
\newcommand{\C}{\ensuremath{\mathbb{C}}}

\DeclareMathOperator{\tb}{tb}

\def\dfn#1{{\em #1}}

\title[Contact embeddings]{Braided embeddings of contact 3--manifolds\\ in the standard contact 5--sphere}

\author{John B. Etnyre}
\address{School of Mathematics \\ Georgia Institute of Technology}
\email{etnyre@math.gatech.edu}
\urladdr{\href{http://www.math.gatech.edu/~etnyre}{http://www.math.gatech.edu/\~{}etnyre}}

\author{Ryo Furukawa}
\address{Department of Mathematics\\ The University of Tokyo} 
\email{furukawa@ms.u-tokyo.ac.jp}

\begin{document}

\begin{abstract}
In this paper we study embeddings of contact manifolds using braidings of one manifold about another. In particular we show how to embed many contact $3$--manifolds into the standard contact $5$--sphere. We also show how to obstruct braidings and branched covers of one manifold over another using contact geometry.
\end{abstract}

\maketitle

\section{Introduction}
It is a classical result of Hirsch \cite{Hirsch61} that any closed oriented 3--manifold can be embedded in $S^5$. An alternate proof of this fact is due to Hilden, Lozano and Montesinos \cite{HildenLozanoMontesinos83}, see also \cite{Hilden78}. We call their technique {\em braided embedding}. In this paper we study general braided embeddings and see how they interact with contact embeddings. In particular we can use contact geometry to obstruct certain braided embeddings and we can use braided embeddings to partially generalize Hirsch's theorem to the contact category and study contact embeddings in general. 

The work in this paper should be thought of as one of the first steps in understanding contact submanifolds of a contact manifold, which in turn should be thought of as a generalization of transverse knot theory in 3--dimensional contact geometry to higher dimensions. Recall in dimension 3, transverse knots can be used to construct contact structures in every homotopy class of plane field on all closed oriented 3--manifolds from surgery on the standard contact 3--sphere \cite{Lutz71, Martinet71} (and indeed all contact structures on these manifolds \cite{Conway14?}), study properties of contact 3--manifolds \cite{Bennequin83}, and they can distinguish all contact structures \cite{EtnyreVanHornMorris10}. It is likely that contact submanifolds will play a similar role in higher dimensions.

\subsection{Contact Embeddings}
We call $(M,\xi)$ a contact submanifold of $(W,\xi')$ if $M$ is a submanifold of $W$ that is transverse to $\xi'$ and $\xi=\xi'\cap TM$. A contact embedding of one contact manifold $(M,\xi)$ into another $(W,\xi')$ is simply an embedding $e\colon M\to W$ such that $(e(M), de(\xi))$ is a contact submanifold of $(W,\xi')$. We notice that this is a generalization of the notion of a transverse knot in a contact 3--manifold (since the contact structure on $S^1$ is simply the trivial vector space in each $T_xS^1$). While transverse knots have been extensively studied in contact 3--manifolds --- in fact, it has been shown that understanding transverse knots in a contact structure determines the contact structure \cite{EtnyreVanHornMorris10} --- there seems to be little known in higher dimensions. 

The most basic questions that can be asked are 

\begin{question}\label{D}
Given a contact $(2n+1)$--manifold $(M,\xi)$ for what $m$ does $(M,\xi)$ contact embed in the standard contact sphere $(S^{2m+1},\xi_{std})$?
\end{question}
\begin{question}\label{E}
Given an embedding of an odd dimensional manifold $M$ in $(S^{2n+1},\xi_{std})$ when can it be isotoped to be transverse to $\xi_{std}$ such that $TM\cap \xi_{std}$ is a contact structure on $M$? We will call such an embedding a {\em transverse contact embedding}.
\end{question}
\begin{question}\label{F}
Given two transverse contact embeddings of $M$ into $(S^{2n+1},\xi_{std})$ when are they isotopic?
\end{question}

Towards Question~\ref{D} the following results is the analog of the Whitney embedding theorem in the contact category.
\begin{thm}[Gromov 1986, \cite{Gromov86}\footnote{In Section~3.4.3 of \cite{Gromov86} the theorem was stated with target space being $S^{6n+3}$, but these the techniques can be improved to give the stated result as discussed in \cite[p. 140]{Torres11}}]
Any contact $(2n+1)$--manifold contact embeds in the standard contact structure on $S^{4n+3}$. 
\end{thm}
\begin{remark}
A more explicit embedding of contact 3--manifolds in $(S^{7},\xi_{std})$ was given by Mori \cite{Mori04} using open books and this proof was generalized by Torres in \cite{Torres11}.
\end{remark}
\begin{remark}\label{highercodim}
In \cite{EliashbergMishachev02} Eliashberg and Mishachev showed that given a closed contact manifold $(M,\xi)$ and a contact manifold $(W,\xi')$ of dimensions $m$ and $n$, respectively, then an embedding $e\colon M\to W$ can be isotoped to a contact embedding if (1) $n\geq m+4$ and (2)  $de\colon TM\to TW$ can homotoped through bundle injections to a bundle map $F\colon TM\to TW$ for which $F^{-1}(\xi')=\xi$ and $F$ gives a conformally symplectic map from $\xi$ to $\xi'$ (with respect to the conformal symplectic structures induced from contact forms). Notice that this implies the answer to Question~\ref{E} reduces to bundle theory, in co-dimension 4 or larger. Similarly Question~\ref{D} (resp.~Question~\ref{F}) reduces to the question of the existence (resp.~the isotopy) of smooth embeddings and a bundle theory question when the co-dimension is at least 4. Below we will primarily be interested in the co-dimension 2 case to which the Eliashberg-Mishachev $h$-principle does not apply. 
\end{remark}

Though some of our results hold in all dimensions, we now primarily focus on the case of embedding contact 3--manifolds.  Given Hirsch's result mentioned above that any closed oriented 3-manifold embeds in $S^5$ one might ask if the above theorem can be improved. In general the answer is no.
\begin{thm}[Kasuya 2016, \cite{Kasuya16}]\label{thm:obstruct}
If $(M,\xi)$ is a co-dimension 2 contact embedding into a co-oriented contact manifold $(W,\xi')$ and $H^2(W;\Z)=0$, then $c_1(\xi)=0$. 
\end{thm}
Since it is well known there are many contact 3--manifolds with non-vanishing first Chern class it is clear they cannot embed into the standard contact structure on $S^5$. This brings up two natural questions.
\begin{question}\label{G}
Given a 3--manifold $M$, does a contact structure $\xi$ on $M$ contact embed in $(S^5,\xi_{std})$ if and only if $c_1(\xi)=0$?
\end{question} 
\begin{question}\label{H}
Is there any contact 5--manifold into which all contact 3--manifolds contact embed?
\end{question} 
Below we will comment on Question~\ref{H} but we now discuss several results that can be proven using braided embeddings techniques (discussed below) that point to a positive answer to  Question~\ref{G}. 
The first concerns contact structures on $S^3$. 
\begin{thm}\label{embeds3s}
Any contact structure on $S^3$ can be embedded in $(S^5, \xi_{std})$ so that it is isotopic to the standard embedding. Moreover there are infinitely many isotopy classes of embeddings of $S^3$ into $S^5$ so that any contact structure on $S^3$ can be realized by a contact embedding in each of these isotopy classes. 
\end{thm}
\begin{remark}
It is obvious that the standard tight contact structure on $S^3$ embeds. In \cite{Mori12}, Mori showed that the overtwisted contact structure $\xi_1$ (see Section~\ref{htpyclasses} for notation) embeds so that it is smoothly isotopic to the standard embedding and hence, using connected sums as in Lemma~\ref{cconnectsum}, it is clear one can embed $\xi_n$ for all $n\geq 1$.  So the real content of the theorem is to embed the $\xi_n$ for $n\leq 0$ and to control the isotopy class of the embeddings. 
\end{remark}

Using the relative contact connected sum lemma, Lemma~\ref{cconnectsum}, we have the immediate corollary. 

\begin{cor}\label{allotifno2}
In every smooth isotopy class of contact embedding of $(S^3,\xi_{std})$ into $(S^5,\xi_{std})$ there is also an embedding of every overtwisted contact structure on $S^3$. \hfill \qed
\end{cor}

\begin{question}
Is there a smooth isotopy class of embedding of $S^3$ in $S^5$ that does not contain an embedding of $(S^3,\xi_{std})$ into $(S^5,\xi_{std})$?
\end{question}

Theorem~\ref{embeds3s} also allows us to show the following result. 
\begin{thm}\label{allhaveembed}
Let $M$ be a 3--manifold with no 2--torsion in its first homology group. Then an overtwisted contact structure $\xi$ on $M$ embeds in $(S^5, \xi_{std})$ if and only if $c_1(\xi)=0$. 
\end{thm}


For some 3--manifolds we can do better and completely answer Question~\ref{G}.  

\begin{thm}\label{embedall}
Let $M$ be one of the following manifolds
\begin{enumerate}
\item a lens space $L(p,q)$ (this includes $S^3$) with $p$ odd or with $p$ even and $q=1$ or $q=p-1$,
\item $S^1\times S^2$, or
\item $T^3$.
\end{enumerate}
A contact structure $\xi$ on $M$ can be embedded in $(S^5, \xi_{std})$ if and only if its first Chern class is zero, $c_1(\xi)=0$. 
\end{thm}
\begin{remark}
The above theorem can be extended to all $L(p,q)$ with $p<10$, but as the proofs in the cases not mentioned in the theorem are {\em ad hoc} but similar to the ones used in the theorem we do not include them here. We expect the same techniques to extend to all $L(p,q)$. We also note that in Lemma~\ref{alltight} we show that the theorem is true for tight contact structures on all lens spaces, so surprisingly the difficultly in proving the theorem is embedding all overtwisted contact structures when there is 2--torsion in the first homology. 
\end{remark}
\begin{remark}
Several of the embeddablity results in this theorem were previously known to Mori. Specifically, he observed that tight contact structures on $T^3$ embed based on the methods in \cite{Mori12}. Also, using open book embeddings, Mori announced that the double branched covers of $(S^3,\xi_{std})$ can be embedded in $(S^5, \xi_{std})$ in a 2012 talk. From this some of the results about embedding contact structures on lens spaces easily follow, though getting the full statements about lens spaces in the theorem takes considerably more work. 
\end{remark}

We notice a slight modification of Question~\ref{G} does have a positive answer. Recall in \cite{BormanEliashbergMurphy15} Borman, Eliashberg and Murphy gave a definition of {\em overtwisted} contact structures in all dimensions and showed there was a unique, up to isotopy, overtwisted contact structure $\xi_{ot}$ on $S^5$. 
\begin{thm}\label{embedinOT}
A contact 3--manifold $(M,\xi)$ contact embeds in $(S^5,\xi_{ot})$ if and only if $c_1(\xi)=0$.
\end{thm}
\begin{remark} This result is a direct corollary of \cite{BormanEliashbergMurphy15} and \cite{Kasuya16}. 
In \cite{Kasuya16}, Kasuya showed that if $(M,\xi)$ is a contact 3--manifold with vanishing first Chern class then it embeds in {\em some} contact structure on $\R^5$, but the contact structure could depend on $\xi$. It was essential in Kasuya's argument that the target space was open (since he relied on an existence result of Gromov for open manifolds). Using \cite{BormanEliashbergMurphy15} one can immediately extend Kasuya's argument to obtain the above result, see Section~\ref{othigh}. 
\end{remark}
\begin{remark}
Question~\ref{G} is still relevant as one would like to embed contact manifolds in the ``simplest" and ``nicest" contact structures possible.  
\end{remark}
In regards to Question~\ref{H} and looking for the ``simplest'' target space for embeddings one might ask if all contact 3--manifolds can be contact embedded in some contact structure on the product of a surface and a 3--manifold. In \cite{EtnyreLekili} the first author and Lekili show that there is an overtwisted contact structure on $S^2\times S^3$ into which every contact 3--manifold contact embeds and there is also a Stein fillable contact structure (and hence not overtwisted) on the twisted $S^3$ bundle over $S^2$ with the same property.

We now turn to Question~\ref{F} concerning the uniqueness of transverse contact embeddings. We note that two invariants of a transverse contact embedding of $M^3$ into $(S^5,\xi_{std})$ are (1) the smooth isotopy class and (2) the contact structure induced on $M$ by the embedding. Thus Theorem~\ref{embeds3s} shows there are infinitely many non-transversely isotopic contact embeddings of $S^3$ (we will abbreviate the phrase ``transverse contact isotopy'' to ``transverse isotopy''). 
\begin{question}
Are there embeddings of $(S^3,\xi_{std})$ into $(S^5,\xi_{std})$ that are smoothly isotopic but not isotopic through contact embeddings?
\end{question}
We do not answer this question here, but think that it is likely there are. We also point out that there are likely no simple algebraic invariants (like the self-linking number in dimension 3) by observing the following result which follows directly from \cite{BormanEliashbergMurphy15}.
\begin{thm}\label{looseweek}
Let $e_i\colon (S^3,\xi)\to (S^5,\xi_{ot}), i=1,2,$ be two contact embeddings of contact structure on $S^3$ into the overtwisted contact structure on $S^5$ such that the contact structure on the complements of their images are overtwisted. If $e_1$ is smoothly isotopic to $e_2$, then there is a contactomorphism $\phi\colon (S^5,\xi_{ot})\to (S^5,\xi_{ot})$ such that $e_2=\phi\circ e_1$. 
\end{thm}
A contact submanifold of an overtwisted contact manifold is called loose if the contact structure on the complement of the submanifold is overtwisted. The above theorem basically says that up to contactomorphism the only invariants of a loose transverse contact embedding of the 3--sphere are the two discussed above.

\subsection{Braided embeddings}
Given $n$--manifolds $Y$ and $M$ we say that $M$ is {\em braided about $Y$} if there is an embedding $e\colon M\to Y\times D^2$ such that $\pi\circ e\colon M\to Y$ is a branched covering map, where $\pi\colon Y\times D^2\to Y$ is projection. (Though the definition works in complete generality, in this paper we will restrict attention to the case where the branch locus is a smooth submanifold.) If $M$ is braided about $Y$ and $Y$ is embedded in some $(n+2)$--manifold $W$ with neighborhood $Y\times D^2$ then notice that there is an embedding of $M$ into $W$ too.  Such an embedding of $M$ into $W$ will be called a {\em braided embedding} and if an embedding of $M$ into $W$ can be isotoped to be such an embedding then we will say the embedding can be braided about $Y$. The most common setting for such problems will be when $Y=S^n$ and $W=S^{n+2}$. 
We have the following obvious questions. 
\begin{question}\label{A}
Which $M$ can be braided about which $Y$?
\end{question}
\begin{question}\label{B}
Can a given branched covering $p\colon M\to Y$ be realized via a braiding of $M$ about $Y$?
\end{question}
\begin{question}\label{C}
Can a given embedding of $M^n$ into $S^{n+2}$ be isotoped to be braided about the standardly embedded $S^n$ in $S^{n+2}$?
\end{question}

Below we will discuss various answers to theses questions and how they are related to, and can be studied by, contact geometry. 

Since a branched cover in dimension 1 is simply a cover, one sees the notion of a braided embedding $S^1$ in $S^3$ simply recovers the classical notion of a closed braid. And thus when $n=1$, Question~\ref{C} was answered affirmatively by Alexander \cite{Alexander23}. When $n=2$, Viro discussed an affirmative answer to Question~\ref{C} in lectures in 1990 where the notion of braided embeddings in the context we are using here seem to have first been discussed (though there were precursors in \cite{Rudolph83a} and \cite{HildenLozanoMontesinos83}). Viro's proof never appeared in the literature but an alternate proof was given by Kamada in \cite{Kamada94}. 

The first partial answers to Question~\ref{B} appeared in work of Hilden \cite{Hilden78} and Hilden, Lozano, and Montesinos \cite{HildenLozanoMontesinos83} where it was shown that dihedral covers of $S^3$ can be braided and in the former paper it was observed that cyclic covers in all dimensions can be braided.
Prompted by discussions with the first author, Questions~\ref{A} and~\ref{B} were further addressed by Carter and Kamada in \cite{CarterKamada12, CarterKamada15b, CarterKamada15}. Concerning Question~\ref{B}, in Theorem~\ref{cyclicbraid} we give another proof of Hilden's result  that cyclic branched covers (with certain conditions on the branched set) can always be realized as braidings. This was also observed for 2--fold covers in dimension 2, 3, and 4 in \cite{CarterKamada15b}. We also show in Theorem~\ref{immerse} that any branched cover (whose branch locus is a smooth submanifold with trivial normal bundle) can be braided by using an immersion instead of an embedding. This generalizes the result for simple 3--fold branched covers in dimension 1, 2 and 3 from \cite{CarterKamada15b}.

An example of a simple branched covering of $S^3$ that could not be braided about $S^3$ (that is an example showing the answer to Question~\ref{B} is not always yes) was given in \cite{CarterKamada15b}. In Example~\ref{nobraid} we give an infinite family (and indicate how to make many more) of such examples using contact geometry (or more precisely the bundle theory underlying contact geometry). 

It is known that any $n$--manifold is a cover of $S^n$ branched along the $(n-2)$--skeleton of a standardly embedded $n$--simplex \cite{Alexander20} and there has been much study of how simple the branched set can be made. This is discussed more in Section~\ref{sec:bc} and in Examples~\ref{trivialnb} and~\ref{trivialnb2} we show how to use Theorem~\ref{immerse} to restrict the possible branched loci for the realization of some manifolds as branched covers over spheres. In particular, we show that $\C P^n$ cannot be realized as a cover of $S^{2n}$ branched over an embedded orientable submanifold for $n>1$. This was previously known for $n=2$ but seems to be a new result for larger $n$.  

Concerning Question~\ref{A} we note that Theorem~\ref{cyclicbraid}, or \cite{CarterKamada15b}, and the well-known fact any oriented 2--manifold is a 2--fold branched cover over $S^2$ says any oriented surface can be braided about $S^2$. The only other result along these lines seems to be the following result. 

\begin{thm}[Hilden, Lozano and Montesinos 1983,  \cite{HildenLozanoMontesinos83}]
Any closed oriented 3--manifold can be braided about $S^3$ and hence has a braided embedding in $S^5$.
\end{thm}

In regards to Question~\ref{E} about when an embedding can be isotoped to be a transverse contact embedding we note that outside of dimension 3, where the answer is known to be {\em yes}, it is not expected that there is a local $h$-principle that could give a positive answer to this question, but there might be a large isotopy providing a positive answer. We have the following result, which follows easily from Theorem~\ref{MainContactBraid}, in dimension 5 relating this to braiding. 

\begin{thm}\label{braidedembedimplycont}
If an embedding $M\to S^5$ can be isotoped to be a braided embedding about the standard $S^3$ in $S^5$ then it can be isotoped to be a transverse contact embedding.
\end{thm}

Thus in dimension 5,  Question~\ref{E}  reduces to Question~\ref{C} for $n=3$. 

{\em Acknowledgments:} The authors are grateful to the anonymous referee for many valuable suggestions on the paper. The authors thank James Conway and Emmy Murphy for useful conversations about the work in this paper. We also thank Patrick Massot for pointing out subtleties that were overlooked in our original formulation of Theorem~\ref{bccontact}. In addition, we thank Donald Davis for help understanding immersions of $\R P^n$ into Euclidean space. The second author also thanks his advisor Takashi Tsuboi for continuing advice. He also thank Atsuhide Mori for useful arguments and Naohiko Kasuya for useful conversations. 
This first author was partially supported by a grant from the Simons Foundation (\#342144) and NSF grants DMS-1309073 and DMS-1608684. The second author was supported by the Program for Leading Graduate Schools, MEXT, Japan.

\section{Background and preliminary results}
In this section we begin by recalling basic definitions and results about contact embeddings. In the following two subsection we review a few facts about, homotopy classes of plane fields, contact 3--manifolds, and transverse and Legendrian knots. In Subsection~\ref{othigh} we discuss overtwisted contact structures in higher dimensions and prove Theorems~\ref{embedinOT} and~\ref{looseweek}. Then in Subsections~\ref{obd} and~\ref{sec:bc} we recall a few definitions and facts about open book decompositions and branched covers, respectively, and we end this section by discussing specific branched covers in dimension 2 and 3. 

\subsection{Contact structures and contact embeddings}

Recall a (co-oriented) contact structure on an oriented $(2n+1)$--dimensional manifold $M$ is a hyperplane distribution $\xi\subset TM$ that is defined as the kernel of a 1--form $\alpha$, $\xi=\ker \alpha$, for which $\alpha\wedge (d\alpha)^n$ is a volume form on $M$ defining the given orientation. 

Given $\xi=\ker\alpha$ one may easily see that $d\alpha$ gives $\xi$ the structure of a symplectic bundle. It is well known, see for example \cite{McDuffSalamon95}, that such a bundle also has a complex structure $J\colon \xi\to \xi$ that is compatible with $d\alpha$ and that $J$ is unique up to homotopy. Thus to a contact structure $\xi$ we can associate its Chern classes $c_1(\xi),\ldots, c_n(\xi)$ which become invariants of $\xi$.

One of the simplest examples of a contact structure is on the unit sphere $S^{2n-1}$ in $\C^n$ and is defined as the set of complex tangencies 
\[
\xi_{std}=TS^{2n-1}\cap J(TS^{2n-1}),
\]
where $J\colon T\C^n\to T\C^n$ is the almost complex structure on the tangent space of $\C^n$ induced from its complex structure. We will call this the \dfn{standard contact structure} on $S^{2n-1}$ and denote it $\xi_{std}$ without reference to the dimension of the sphere, which should always be clear from context.

Given two contact manifolds $(M,\xi)$ and $(W,\xi')$ we will call an embedding $e\colon M\to W$ a contact embedding if $e$ is transverse to $\xi'$ and a contactomorphism from $(M,\xi)$ to $(e(M), Te(M)\cap \xi')$. This can equivalently be expressed by saying there is a contact form $\alpha'$ defining $\xi'$ such that $e^*\alpha'$ is a contact form defining $\xi$. A simple and well-known application of a Moser type argument, see for example \cite[Proposition~2.1]{EtnyrePancholi11} or \cite[Theorem~2.5.15]{Geiges08}, yields the following result. 
\begin{prop}\label{nbhdprop}
Suppose that $e\colon (M^{2n+1},\xi)\to (W^{2(n+k)+1},\xi')$ is a contact embedding for which $e(M)$ has trivial normal bundle (as a symplectic bundle with symplectic structure induced from a contact form). Then there is a neighborhood of $e(M)$ in $W$ that is contactomorphic to $M\times D^{2k}$ with the contact structure $\ker(\alpha+\sum_{j=1}^k r_j^2\, d\theta_j),$ where $(r_j,\theta_j), j=1,\ldots, k$ are polar coordinates on the $D^2$ factors of $D^{2k}$ and $\alpha$ is a contact form for $\xi$. 
\end{prop}
\begin{remark}
Our applications of this proposition will be in the co-dimension 2 case where having trivial oriented normal bundle is equivalent to having trivial symplectic normal bundle. 
\end{remark}

We also make the simple observation that the connect sum operation, which is well known in the contact category, can be done in a relative setting.
\begin{lem}\label{cconnectsum}
If $(M_i,\xi_i)$ is a contact submanifold of $(W_i,\xi'_i)$ for $i=1,2$, then $(M_1\# M_2,\xi_1\# \xi_2)$ is a contact submanifold of $(W_1\# W_2,\xi'_1\# \xi'_2)$. 
\end{lem}

\subsection{Homotopy classes of plane fields on 3--manifolds}\label{htpyclasses}
The most basic invariant of a contact structure on a 3--manifold is the homotopy type of the underlying plane field. We now review part of the classification of homotopy classes of plane fields on a closed $3$--manifold as described in \cite{Gompf98}. 

Let $\xi$ be an oriented $2$--plane field on a closed oriented $3$-manifold $M$. Any two such plane fields are homotopic over the 1--skeleton of $M$. The homotopy type of $\xi$ over the 2--skeleton is completely determined by a refinement of $c_1(\xi)$. Namely let $\mathcal{S}$ be the set of spin structures on $M$ and $\mathcal{H}$ be the subset of $H_1(M)$ consisting of classes $c$ such that $2c$ is Poincar\'e dual to $c_1(\xi)$. In \cite{Gompf98}, Gompf defines a map
\[
\Gamma_\xi\colon \mathcal{S}\to\mathcal{H}
\]
that completely determines $\xi$, up to homotopy, over the 2--skeleton of $M$. Notice that if $H^2(M)$ (or equivalently $H_1(M)$) has no 2--torsion then $\mathcal{H}$ has a unique element in it and $\Gamma_\xi$ is completely determined by $c_1(\xi)$. 

If $\xi$ has torsion $c_1(\xi)$ in $H^2(M)$ then the homotopy class of $\xi$ over the 3--skeleton (and hence over $M$) is determined by a ``3--dimensional obstruction" $d_3(\xi)$ (and of course $\Gamma_\xi$), see \cite[Definition~4.15]{Gompf98}. In order to define $d_3(\xi)\in \Q$ it was shown in \cite{Gompf98} that one may choose an almost complex $4$-manifold $(X,J)$ whose almost complex boundary is $(M,\xi)$, then one defines
\[
d_3(\xi)=\frac{1}{4}(c_1^2(X,J)-3\sigma(X)-2(\chi(X)-1)),
\]
where $\sigma(X)$ and $\chi(X)$ are the signature of $X$ and the Euler characteristic of $X$, respectively. Notice that we have subtracted 1 from $\chi(X)$ unlike the definition in \cite{Gompf98}. This is done so that on $S^3$ the invariant $d_3$ takes values in $\Z$ instead of the half-integers and there is a better connected sum formula. One may easily check that for the standard contact structure on $S^3$ we have $d_3(\xi_{std})=0$. 

\begin{prop}[Gompf 1998, {\cite[Theorem~4.16]{Gompf98}}]\label{gompfclass}
Let $\xi_1$ and $\xi_2$ be $2$-plane fields on a closed oriented $3$-manifold $M$ and suppose that $c_1(\xi_1)$ and $c_1(\xi_2)$ are torsion classes. Then $\xi_1$ and $\xi_2$ are homotopic if and only if $\Gamma_{\xi_1}({\mathbf s})=\Gamma_{\xi_2}({\mathbf s})$ for a (and hence any) spin structure ${\mathbf s}$ on $M$ and $d_3(\xi_1)=d_3(\xi_2)$.   
\end{prop}

From the formula above for the $d_3$ invariant it is clear that if $\xi$ and $\xi'$ are two contact structures with torsion first Chern classes then the $d_3$ invariant of the connect sum of the contact manifolds is
\[
d_3(\xi\#\xi')=d_3(\xi)+d_3(\xi').
\]

\subsection{Contact structures on 3--manifolds and transverse and Legendrian knots}\label{contactintro}
Recall that contact structures on 3--manifolds fall into one of two types: tight or overtwisted. A contact structure $\xi$ on a 3--manifold $M$ is \dfn{overtwisted} if there is an embedded disk $D\subset M$ such that $D$ is tangent to $\xi$ along its entire boundary: $T_xD=\xi_x$ for all $x\in \partial D$. Such a disk is called an \dfn{overtwisted disk}. If no such disk exists then we call $\xi$ \dfn{tight}. It is well known that $\xi_{std}$ on $S^3$ is tight \cite{Bennequin83} and the unique tight contact structure on $S^3$ \cite{Eliashberg92a}.

In \cite{Eliashberg89}, Eliashberg classified overtwisted contact structures. 
\begin{thm}[Eliashberg 1989, \cite{Eliashberg89}]\label{otclass}
The inclusion of the set of overtwisted (co-oriented) contact structures on a closed oriented 3--manifold into the set of oriented plane fields induces a one-to-one correspondence of connected components. 
\end{thm}
From Proposition~\ref{gompfclass} and this theorem we know that for each integer $n\in \Z$ there is an overtwisted contact structure $\xi_n$ on $S^3$ with $d_3(\xi_n)=n$. 

Given a null-homologous transverse knot $K$ in a contact 3--manifold $(M,\xi)$ it has a simple topological invariant called the self-linking number. Since $K$ is null-homologous it bounds an embedded surface $\Sigma$ and the restriction of $\xi$ to $\Sigma$ is trivial so we can choose a non-zero section $v$ of $\xi$ restricted to $\Sigma$. Let $K'$ be a copy of $K$ pushed slightly in the direction of $v$. Then the \dfn{self-linking number} of $K$, $sl(K)$, is simply the linking of $K$ and $K'$ (or equivalently the signed intersection of $K'$ and $\Sigma$). One may also see $sl(K)$ as (the negative of) the obstruction to extending the outward pointing vector field along $K$ to a non-zero vector field in $\xi|_\Sigma$. From this description it is clear that $sl(K)$ is independent of choices if $c_1(\xi)=0$. 

Bennequin showed \cite{Bennequin83} that any transverse knot in $(S^3,\xi_{std})$ (which we think of as $\R^3=S^3-\{p\}$ and $\xi_{std}=\ker(dz+r^2\, d\theta)$) can be written as the closure of a braid. If a transverse knot $K$ is written as the closure of an $n$--braid $\sigma$ then 
\begin{equation}\label{slformula}
sl(K)=writhe(\sigma)-n,
\end{equation}
where $writhe(\sigma)$ is the writhe of the obvious projection of the braid. 

One may stabilize a transverse knot. This is a local operation that reduces the self-linking number by 2. If $K$ is represented by an $n$-braid then the transverse stabilization can be seen as a negative braid stabilization, that is add a strand to the braid and multiply the braid word by the inverse of the standard generator $\sigma_n$. 
For more details on this and transverse knots in general see \cite{Etnyre05, Geiges08}.

Recall that a knot $K$ in a contact 3--manifold $(M,\xi)$ is Legendrian if it is everywhere tangent to $\xi$. We assume that $K$ is null-homologous and so has a canonical (Seifert) framing. The contact structure also gives $K$ a framing and the difference between this and the Seifert framing is an integer called the \dfn{Thurston-Bennequin invariant} of $K$ and denoted $tb(K)$. Orienting $K$ we can discuss the Euler class of $\xi$ relative to an oriented vector field along $K$. Evaluating this on the Seifert surface results in an integer that is the \dfn{rotation class} of $K$, which is denoted by $r(K)$. It is well known that Legendrian knots in the standard contact structure on $\R^3$ (or $S^3$) can be represented by their \dfn{front projection}, see \cite{Etnyre05, Geiges08}. Moreover, a Legendrian knot $K$ can be \dfn{stabilized} in a positive and a negative way which we denote $S_+(K)$ and $S_-(K)$, respectively. In the front projection this just amounts to ``adding zigzags'' and we know $tb(S_\pm(K))=tb(K)-1$ and $r(S_\pm(K))=r(K)\pm 1$. 

Given $K$ in $(M,\xi)$ one can perform $tb(K)\pm1$ surgery on $M$ to get a manifold $M_K(\tb(K)\pm 1)$ and there is a unique contact structure $\xi'$ on it that agrees with $\xi$ on the complement of the surgery torus and is tight on the surgery torus. We say $(M_K(\tb(K)\pm 1),\xi')$ is obtained from $(M,\xi)$ by \dfn{($\pm 1$)-contact surgery} on $K$. We also call $(-1)$-contact surgery \dfn{Legendrian surgery}. The main result we will need below is the following.
\begin{thm}[Eliashberg 1990, \cite{Eliashberg90a}; Gompf 1998, \cite{Gompf98}]\label{buildstein}
Given a Legendrian link $K_1\cup\ldots\cup K_n$ in $(S^3,\xi_{std})$, then the manifold $X$ obtained from $B^4$ by attaching 2--handles to the link with framings $tb(K_i)-1$ has the structure of a Stein domain with first Chern class
\[
c_1(X)=\sum_{i=1}^n r(K_i) h_i,
\]
where  $h_i$ is Poincar\'e dual to the co-core of the handle attached to $K_i$. Moreover, the complex tangencies to the boundary give a contact structure obtained from $(S^3,\xi_{std})$ by Legendrian surgery on the link. 
\end{thm}
Recall a Stein manifold is a complex manifold with a proper embedding in $\C^N$ for some large $N$. The sub-level set of a regular value of the restriction of the radial function on $\C^N$ to the Stein manifold will be called a Stein domain. It is well known that the contact structure induced on the boundary of a Stein domain is tight. 

It is also shown in \cite{Gompf98} how to compute the $\Gamma$ invariant of contact structures obtained through Legendrian surgery. To state this recall, see \cite{Gompf98}, that if $M$ is obtained from $S^3$ by surgery on some link $L=K_1\cup,\ldots, \cup K_n$, with surgery framing $a_i$ on link component $K_i$,  then spin structures on $M$ are in one-to-one correspondence with characteristic sub-links of $L$. A sub-link $L'$ of $L$ is called characteristic if for each $K_i$ in $L$ we have $a_i\equiv\text{ linking} (K_i,L')  \mod 2$. Moreover, if $\gamma_1,\ldots, \gamma_k$ represent a basis for the homology of $M$ then a spin structure is characterized by specifying the framings (modulo $2$) on the $\gamma_i$ with which we can attach a 2--handle and extend the spin structure over the handle. If a spin structure is given by a characteristic sub-link $L'$ then this framing is given by $\text{ linking} (\gamma_i,L')  \mod 2$.

Now suppose $L=K_1\cup \ldots \cup K_n$ is a Legendrian link in $(S^3,\xi_{std})$ and $(M,\xi)$ the contact manifold obtained by Legendrian surgery on this link. Let $L'$ be a characteristic sub-link of $L$ corresponding to the spin structure $\mathbf s$. Then 
\begin{equation}\label{gammainvt}
\Gamma_\xi({\mathbf s})=\frac 12 \sum_{i=1}^n \left( r(K_i) + \text{linking}(K_i,L')\right) \mu_i,
\end{equation}
where $\mu_i$ is the homology class determined by the meridian of $K_i$. 

\subsection{Overtwisted contact structures in higher dimensions}\label{othigh}
In \cite{BormanEliashbergMurphy15}, Borman, Eliashberg, and Murphy introduced the notion of an overtwisted contact structure in all dimensions. There definition of overtwisted is a bit difficult to state but in \cite{CasalsMurphyPresas}, Casals, Murphy, and Presas gave alternate characterizations of overtwistedness and we present one of those here. 

Consider $P=Z\times D^2$ in $T^*S^{n-1}\times \R^3$ where $Z$ is the zero section of $T^*S^{n-1}$ and $D^2$ is the disk of radius $\pi$ in the $z=0$ plane in $\R^3.$ Let $\xi'=\ker (\lambda+ \cos r\, dz+ r\sin r\, d\theta)$, where $\lambda$ is the Liouville 1--form on $T^*S^{n-1}$ and $(r,\theta,z)$ are cylindrical coordinates on $\R^3$. We call a contact structure $\xi$ on a $(2n+1)$--dimensional manifold $M$ \dfn{overtwisted} if there is an embedding of the germ of the contact structure $\xi'$ along $P$ in $T^*S^{n-1}\times \R^3$ such that the image of $P$ is contained in an open ball in $M$ and the image of an open Legendrian submanifold $Z\times \Lambda_0$, where $\Lambda_0$ is an open leaf of the characteristic foliation of $D^2\subset (\R^3\cap \{r<\pi, z=0\}, \ker (\cos r\, dz+ r\sin r\, d\theta))$, has relative rotation number zero with respect to a punctured Legendrian disk. The image of $P$ is typically called a \dfn{small plastikstufe with spherical core and rotation 0}. See \cite{CasalsMurphyPresas} for more details on the definition. 

We also recall that an \dfn{almost contact structure} on a $(2n+1)$--dimensional manifold $M$ is a reduction of the structure group of the tangent bundle of $M$ to $U(n)\times {\bf {1}}$ and so correspond to sections of the $SO(2n+1)/U(n)$-bundle associated to the tangent bundle of $M$. From this one can see that in dimension 5 the only obstruction to the existence of an almost contact structure on a manifold $M$ is in $H^3(M;\Z)$ and the only obstruction to homotoping one almost contact structure to another is in $H^2(M,\Z)$, see for example \cite[Section~8.1]{Geiges08}.

The main theorems from \cite{BormanEliashbergMurphy15} that we will need are the following.
\begin{thm}[Borman, Eliashberg and Murphy 2014,  \cite{BormanEliashbergMurphy15}]\label{BEM1}
Let $M$ be a $(2n+1)$--dimensional manifold and $A$ a closed subset of $M$. If $\eta$ is an almost contact structure on $M$ that is an actual contact structure on some neighborhood of $A$ then $\eta$ is homotopic rel $A$ to an actual (overtwisted) contact structure on $M$.
\end{thm}

\begin{thm}[Borman, Eliashberg and Murphy 2014,  \cite{BormanEliashbergMurphy15}]\label{BEM2}
Let $M$ be a $(2n+1)$--dimensional manifold and $A$ a closed subset of $M$.
If $\xi$ and $\xi'$ are two contact structure on $M$ that agree on an open neighborhood of $A$ and  are overtwisted and homotopic through almost contact structures when restricted to $M-A$, then they are isotopic as contact structure by an isotopy fixed on $A$. In particular, there is a contactomorphism from $\xi$ to $\xi'$ that is the identity map on $A$. 
\end{thm}

The first theorem says there is an overtwisted contact structure on $S^5$ and computations of the set of almost contact structures on $S^5$, see \cite{BormanEliashbergMurphy15}, together with the second theorem says there is the unique overtwisted contact structure up to isotopy, we denote it by $\xi_{ot}$. 

The above two theorems imply that questions about contact embeddings into overtwisted contact manifolds reduce to questions about smooth embeddings and almost contact structures, the latter is a problem in algebraic topology that can frequently be solved. 
From this observation one may prove Theorem~\ref{embedinOT} which says a contact structure on a 3--manifold embeds in $(S^5,\xi_{ot})$ if and only if its first Chern class vanishes. 
\begin{proof}[Proof of Theorem~\ref{embedinOT}]
Given a contact structure $\xi=\ker \alpha$ on a 3--manifold $M$ with $c_1(\xi)=0$ and an embedding of $M$ into the $5$--ball $B^5$, Kasuya in \cite{Kasuya16} shows how to extend the contact structure $\ker (\alpha + r^2\, d\theta)$ on the neighborhood $M\times D^2$ of $M$ in $B^5$ to an almost contact structure on $B^5$. 
Since $S^5$ is obtained from $B^5$ by attaching a $5$--handle and $\pi_4(SO(5)/U(2))=0$, the almost contact structure extends over $S^5$. Thus Theorem~\ref{BEM1} allows us to homotope this almost contact structure relative to a neighborhood of $M$ to an actual overtwisted contact structure, that is $\xi_{ot}$. Thus creating a contact embedding of $(M,\xi)$ into $(S^5,\xi_{ot})$. 
\end{proof}

Recall that Theorem~\ref{looseweek} says that up to contactomorphism the only invariants of a loose transverse contact embedding are the smooth isotopy class and the induced contact structure. 
\begin{proof}[Proof of Theorem~\ref{looseweek}]
Let $e_i\colon (S^3,\xi)\to (S^5,\xi_{ot}), i=1,2,$ be two (smoothly) isotopic contact embeddings whose images have overtwisted complements. Since they are isotopic (and have 2 dimensional trivial normal bundles) we can use Proposition~\ref{nbhdprop} to find a smooth isotopy $\phi_t\colon S^5\to S^5$ such that $\phi_1$ is a contactomorphism from a neighborhood $N_1$ of the image of $e_1$ to a neighborhood $N_2$ of the image of $e_2$  and so that $\phi_1\circ e_1=e_2$. Since $\overline{S^5-N_i}$ is a homology $S^1\times D^4$  we know that $H^2(\overline{S^5-N_i}, \partial(\overline{S^5-N_i});\Z)=0$. Thus from the discussion of homotoping almost contact structures above we see that $(\phi_1)_*\xi_{ot}$ is homotopic to $\xi_{ot}$ on $S^5-N_2$. By Theorem~\ref{BEM2} we see that they are isotopic relative to $N_2$. Now Gray's theorem gives an isotopy $\psi_t$, rel $N_2$, such that $\phi=(\psi_1\circ\phi_1)$ is a contactomorphism of $\xi_{ot}$ and satisfies $\phi\circ e_1=e_2$. 
\end{proof}

\subsection{Open book decompositions}\label{obd}
Let $M$ be a closed $n$--dimensional manifold. An \dfn{open book decomposition} of $M$ is a pair $(B,\pi)$ where $B$ is a closed $(n-2)$--dimensional submanifold of $M$ and $\pi\colon (M- B)\to S^1$ is a locally trivial fibration such that $\pi^{-1}(\theta)$, $\theta\in S^1$, is the interior of a compact hypersurface $\Sigma_{\theta}$ in $M$ and $\partial \Sigma_{\theta}=B$. We call $B$ the \dfn{binding} and each $\overline{\pi^{-1}(\theta)}$ a \dfn{page} of the open book decomposition $(B,\pi)$.

Following Giroux \cite{Giroux02} we say a contact structure $\xi$ on $M$ is \dfn{compatible} with, or \dfn{supported} by, the open book decomposition $(B,\pi)$ if $\xi$ is isotopic to a contact structure defined by the kernel of a 1-form $\alpha$ such that $\alpha$ is a contact form on $B$ and $d\alpha$ is a symplectic form on each page of the open book. (We notice that $\alpha$ will orient $B$ and $d\alpha$ will orient the pages of the open book, we require that, with these orientations, $B$ is the oriented boundary of the pages.) 
\begin{example}\label{stdobd}
The standardly embedded $S^{n-2}$ in $S^n$ is the binding of an open book and when $n$ is odd this open book supports the standard contact structure $\xi_{std}$. 
\end{example}

\begin{thm}[Thurston-Winkelnkemper 1975, \cite{ThurstonWinkelnkemper75} and Giroux 2002, \cite{Giroux02}]
An open book decomposition of a closed 3--manifold supports a unique contact structure up to isotopy. 
\end{thm}
The analogous theorem is not true in higher dimensions, but in \cite{Giroux02} Giroux gives conditions which guarantee that a high dimensional open book supports a contact structure. He also has poven the following result.
\begin{thm}[Giroux 2002, \cite{Giroux02}]
Every (co-oriented) contact structure on a closed oriented $(2n+1)$--manifold is supported by some open book decomposition. 
\end{thm}

Restricting to dimension 3 we consider the ``extrinsic view" of open book decompositions. Given a pair $(\Sigma, \phi)$ where $\Sigma$ is a surface with boundary and $\phi\colon \Sigma\to \Sigma$ is a diffeomorphism of $\Sigma$ that is equal to the identity near $\partial \Sigma$ then we can build a 3--manifold $M_{(\Sigma,\phi)}$ by gluing copies of $S^1\times D^2$ to the boundary components of the mapping torus  
\[
T_\phi= (\Sigma\times [0,1])/(x,1)\sim (\phi(x),0)
\]
by a diffeomorphisms that sends $S^1\times \{p\}$ to a component of $\partial \Sigma\times\{p'\}$ and $\{q\}\times \partial D^2$ to $\{q'\}\times[0,1]/\sim$. One may easily check that the cores of the glued in tori form a link $B$ that is the binding of some open book $(B,\pi)$ for $M_{(\Sigma,\phi)}$ whose pages are diffeomorphic to $\Sigma$. So according to the above theorems there is a unique contact structure $\xi_{(\Sigma,\phi)}$ on $M_{(\Sigma,\phi)}$ associated to $(\Sigma,\phi)$. We say $(\Sigma,\phi)$ supports a contact structure $\xi$ on a 3--manifold $M$ if there is a contactomorphism from $(M_{(\Sigma,\phi)}, \xi_{(\Sigma,\phi)})$ to $(M,\xi)$.  
See \cite{Etnyre06} for more details. 

Given an open book $(\Sigma,\phi)$ supporting some contact structure $\xi$ on $M$ we can form a new open book by stabilizing. Specifically given a properly embedded arc $\gamma$ in $\Sigma$ let $\Sigma'$ be the result of attaching a 1-handle to $\Sigma$ along $\partial \gamma$. Let $c$ be the embedded curve in $\Sigma'$ obtained by taking the union of $\gamma$ and the core of the added $1$--handle. The open book $(\Sigma', \tau_c\circ \phi)$ is said to be obtained from $(\Sigma,\phi)$ by a \dfn{(positive) stabilization}, where $\tau_c$ is a right handed Dehn twist about $c$. We say $(\Sigma', \tau_c^{-1}\circ \phi)$ is the result of a negative stabilization of $(\Sigma, \phi)$. 

One may check, or see \cite{Etnyre06, Giroux02}, that the smooth manifold described by any stabilization of $(\Sigma,\phi)$ is still $M$. If one does a positive stabilization then the contact structure is also unchanged, but if one does a negative stabilization then the supported contact structure is overtwisted and homotopic to the result of connect summing $(M,\xi)$ with $(S^3, \xi_1)$, where $\xi_1$ is the overtwisted contact structure on $S^3$ with $d_3(\xi_1)=1.$

\subsection{Branched covers}\label{sec:bc}
A map  $p:M\to Y$ is called a \dfn{branched covering} with branch locus $B\subset Y$ if the set of points $\widetilde{B}'$ at which $p$ is not locally injective is precisely the singular locus of $p$ and $B=p(\widetilde{B}')$ is a co-dimension 2 sub-complex of $Y$ such that $p$ restricted to $M-\widetilde{B}$ is a covering map $(M-\widetilde{B})\to (Y-B)$, where $\widetilde{B}=p^{-1}(B)$. Along the top dimensional strata of $\widetilde{B}'$ it is well known that a local model for $p$ is given by 
\[
D^{n-2}\times D^2\to D^{n-2}\times D^2: (x,z)\mapsto (x,z^k),
\]
where we think of $D^2$ as the unit disk in $\C$ and $k$ is an integer larger than 1. We call $k$ the degree, or order, of \dfn{ramification}. Any point outside of $\widetilde{B}'$ is called \dfn{unramified}.  We call the branched cover $n$--fold if the covering map $p$ restricted to $M-\widetilde{B}$ is an $n$--fold covering map. We similarly apply adjectives for covering maps to branched coverings too ({\em e.g.\ }regular, irregular, cyclic, {\em etc.}). An $n$--fold branched covering is called \dfn{simple} the pre-image of any point in $Y$ has either $n$ or $n-1$ points.

In this paper we will restrict ourselves to branched covers where the branch locus $B$ is a smooth submanifold, which implies that $\widetilde{B}$ is too and that $p$ restricted to $\widetilde{B}$ is a covering map. This is a common restriction, but we give the general definition to recall the famous theorem of Alexander \cite{Alexander20} that says a closed oriented $n$--manifold is ($PL$ equivalent to) a covering of $S^n$ branched along the $(n-2)$--skeleton of a standardly embedded $n$--simplex. And there has been much study as to whether the branch locus can be taken to be a smooth submanifold and if so whether it can be assumed to be orientable. In \cite{BersteinEdmonds78} it was shown that the branch locus does not always have to be a smooth submanifold. It is well known  that in dimensions 2, 3 and 4, the branched set can be made smooth, but in dimension 4 one must allow non-orientable surfaces for the branch locus \cite{Piergallini95, Viro84}. We will see in Subsection~\ref{sec:immersions} that one can use the techniques of braided embeddings/immersions to get restrictions on properties of the branch locus necessary to realize certain manifolds. 

There is a well-known construction of contact structures via branched coverings, \cite{Geiges97, Gonzalo87}. We recall the construction here. 
\begin{thm}[Geiges 1997, \cite{Geiges97}, {\"O}zt{\"u}rk and Niederkr{\"u}ger 2007, \cite{OzturkNiederkruger07}]\label{bccontact}
Suppose that $p\colon M\to Y$ is a cover branched along a smooth submanifold $B\subset Y$. Further assume that $Y$ has a contact structure $\xi=\ker \alpha$ such that $B$ intersects $\xi$ transversely and $\xi\cap TB$ is a contact structure on $B$. Then there is a unique (up to isotopy) contact structure $\xi_B$ on $M$ that is given by a contact form $\beta_1$ that can be connected to $\beta_0=p^*\alpha$ by a path $\beta_t$, $t\in [0,1]$, such that $\beta_t$ is a contact form for $t>0$ and $d\!\left(\frac{\partial \beta_t}{\partial t}|_{t=0}\right)$ restricts to a positive form on each (naturally oriented) fiber of the normal bundle of the branch locus in $M$.
\end{thm}
\begin{proof}
Let $\alpha$ be a contact form for $\xi$. It is clear that $p^*\alpha$ is a contact form in the complement of $\widetilde{B}'$ (recall this is the set of point in $M$ where $p$ is ramified). Let $N$ be a tubular neighborhood of $\widetilde{B}'$ in $M$. This is a $D^2$-bundle over $\widetilde{B}'$. Let $\beta$ be the pull back of a connection 1-form on the circle bundle $\partial N$ to $N$ minus the zero section. Also denote by $r\colon N\to \R$ the radial function on $N$. One may easily check that $r^2\beta$ may be extended to a 1-form on $N$. Now let $\eta$ be any 1-form on $M$ for which $d\eta$ agrees with a positive multiple of $d(r^2\beta)$ along $\widetilde{B}'$. We claim that $\alpha_R=p^*\alpha+R\eta$ is a contact from for all sufficiently small $R>0$. Indeed if $M$ is $2n+1$ dimensional then $\alpha_R\wedge (d\alpha_R)^n$ is
\[
p^*(\alpha\wedge (d\alpha)^n)+ R\left[\left(p^*((d\alpha)^n)\wedge \eta\right) +np^*(\alpha\wedge(d\alpha)^{n-1})\wedge d\eta +\eta'(R) \right],
\]
for some form $\eta'(R)$ each summand of which contains positive powers of $R$.
Clearly the first term is a positive multiple of the volume form on $M$ away from $\widetilde{B}'$, so for $R$ sufficiently small this is a contact form on the complement of any sufficiently small neighborhood of $\widetilde{B}'$. On $\widetilde{B}'$ we know $p$ has rank $2n-1$. Thus along $\widetilde{B}'$ the first two terms vanish. Since $p$ restricted to $\widetilde{B}'$ is a covering map we know $p^*(\alpha\wedge(d\alpha)^{n-1})$ is a volume form on $\widetilde{B}'$. On each fiber of the normal bundle $d\eta$ agrees with a positive multiple of $d(r^2\beta)=2r\, dr\wedge \beta+ r^2 d\beta$ and so is an area form on the fiber at $r=0$. Thus the third term in $\alpha_R\wedge (d\alpha_R)^n$ is positive on $\widetilde{B}'$. Hence $\left(p^*((d\alpha)^n)\wedge \eta\right) +n\left(p^*(\alpha\wedge(d\alpha)^{n-1})\wedge d\eta\right)+ \eta'(R)$ is a volume form on a sufficiently small neighborhood of $\widetilde{B}'$ and so is rescaling by any sufficiently small $R>0$. Thus for all small $R>0$ we have established that $\alpha_R$ is a contact form. 


The uniqueness is proved similarly and a detailed proof can be found in \cite{OzturkNiederkruger07}, we also note that Patrick Massot has shown the authors a proof of a stronger result that the space of contact structures induced on a branched cover is contractible.
\end{proof}

\subsection{Branched covers in dimensions 2 and 3}
We first consider branched covers of $D^2$, thought of as the unit disk in $\R^2$. Fix $n$-points $x_1,\ldots, x_n$ in $D^2$ along the $y$-axis (so their $y$-coordinates are increasing with the index). A $k$-fold cover of $D^2$ branched along the $x_i$'s is determined by the ordinary cover of $D^2-\{x_1,\ldots, x_n\}$ which in turn is determined by the monodromy representation of the cover
\[
m\colon \pi_1(D^2-\{x_1,\ldots, x_n\})\to S_k,
\]
where $S_k$ is the symmetric group on $k$ elements. Specifically given a cover $p\colon \Sigma\to (D^2-\{x_1,\ldots, x_n\})$, then label the points $q_1,\ldots q_k$ lying above the base point $x_0$ of $D^2-\{x_1,\ldots, x_n\}$ and for each $[\gamma]\in  \pi_1(D^2-\{x_1,\ldots, x_n\})$ lift $\gamma$ to a path $\widetilde{\gamma}\colon [0,1]\to \Sigma$ starting at $q_i$ and define $m([\gamma])(i)$ to be the index of $\widetilde{\gamma}(1)$. Since there is a one-to-one correspondence between generators of the free group $\pi_1(D^2-\{x_1,\ldots, x_n\})$ and the points $x_1,\ldots, x_n$, we can describe a cover by labeling the marked points with an element of $S_k$. 
\begin{example}\label{exsurfaces}
We show the 2--fold branched cover of $D^2$ branched along two points on the left hand side of Figure~\ref{fig:brachedex}. On the right hand side we give the 3--fold simple cover of $D^2$ branched along 4 points which results in a planar surface $\Sigma$ with three boundary components. 
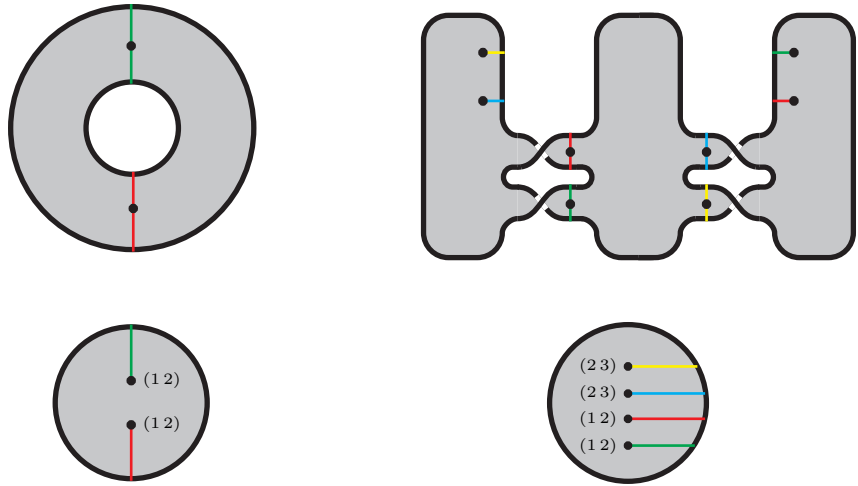
\begin{figure}[htb]{\tiny
\begin{overpic}
{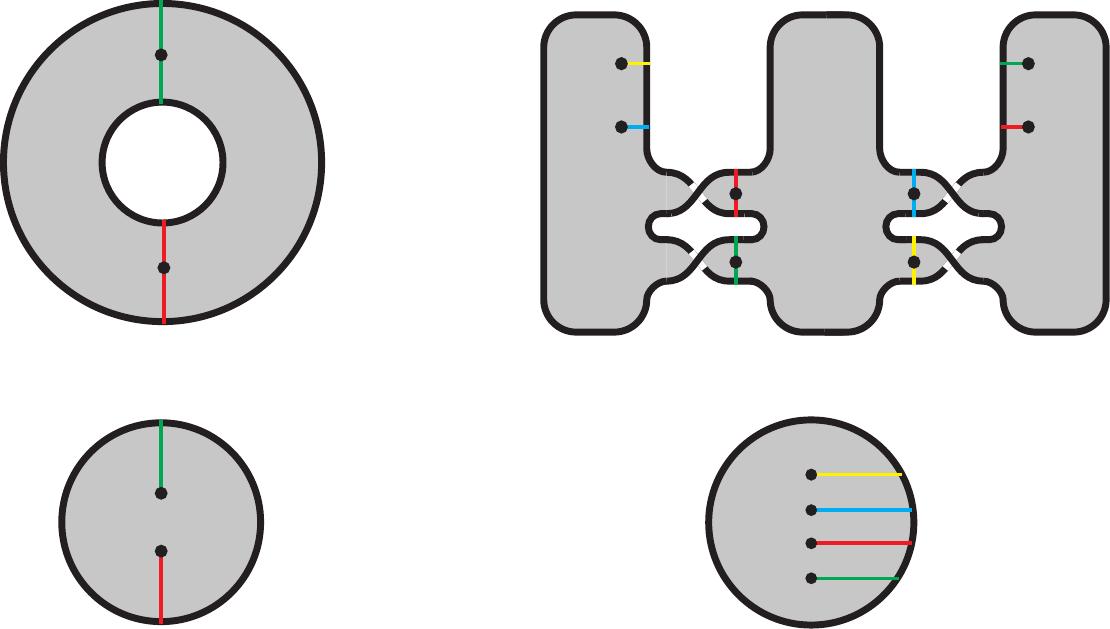}
\put(51,21){$(1\, 2)$}
\put(51,38){$(1\, 2)$}
\put(216, 13){$(1\, 2)$}
\put(216, 23){$(1\, 2)$}
\put(216, 33){$(2\, 3)$}
\put(216, 43){$(2\, 3)$}
\end{overpic}}
\caption{The 2--fold branched cover of the annulus over the disk on the left. The permutations labeling the points on the bottom left describe which ``sheets" are connected as one goes around the branched point. So called ``branched cuts" are also draw to aid in visualizing the cover. On the right hand side one sees the 3--fold simple branched cover of a planar surface with three boundary components $\Sigma$ over the disk. }
\label{fig:brachedex}
\end{figure}
\end{example}

We now turn to the 3 dimensional case. As for surfaces a $k$--fold branched covering $p\colon M\to Y$ will be determined by an ordinary covering of the complement of the branch locus $B\subset Y$ which in turn is determined by a monodromy representation $m\colon \pi_1(Y-B)\to S_k$. If we are branching over $S^3$ then $\pi_1(S^3-B)$ is generated by meridians of $B$. So the monodromy just assigns an element of $S_k$ to each strand in a diagram of $B$ so that they respect the ``Wirtinger relations" at the crossings. Moreover such an assignment will define a monodromy and hence a branched cover. 

Now recall that any transverse link $K$ in $(S^3,\xi_{std})$ can be realized as a closed braid \cite{Bennequin83}. In terms of open books this just means that $K$ is transverse to all the pages of the ``standard open book" from Example~\ref{stdobd}. Notice that given a branched cover $p\colon M\to S^3$ branched along $K$ there is an induced branched cover of each page of the open book. So the open book of $S^3$ with $D^2$ pages lifts to an open book of $M$. Conversely, recall a braid can be described by a diffeomorphism $b\colon D^2\to D^2$ with $n$ marked points. Now given a cover $p\colon \Sigma\to D^2$ branched along the marked points and a diffeomorphism $\widetilde{b}\colon \Sigma\to \Sigma$ such that $p\circ \widetilde{b}=b\circ p$, the open book $(\Sigma,\widetilde{b})$ describes a manifold $M$ that is a cover of $S^3$ branched along the closed braid described by $b$. With a little more thought one sees that the contact structure induced on the cover from Theorem~\ref{bccontact} is supported by $(\Sigma,\widetilde{b})$, see \cite{Casey13, Giroux02}. 
\begin{example}
Consider the 2--fold cover in Example~\ref{exsurfaces}. The diffeomorphism $b\colon D^2\to D^2$ that exchanges the two marked points by a right handed half twist is covered by a right handed Dehn twist $\widetilde{b}$ about the core of the annulus. The closure of the braid corresponding to the diffeomorphism $b$ is shown on the left hand side of Figure~\ref{braidex}. 
\begin{figure}[htb]{\tiny
\begin{overpic}
{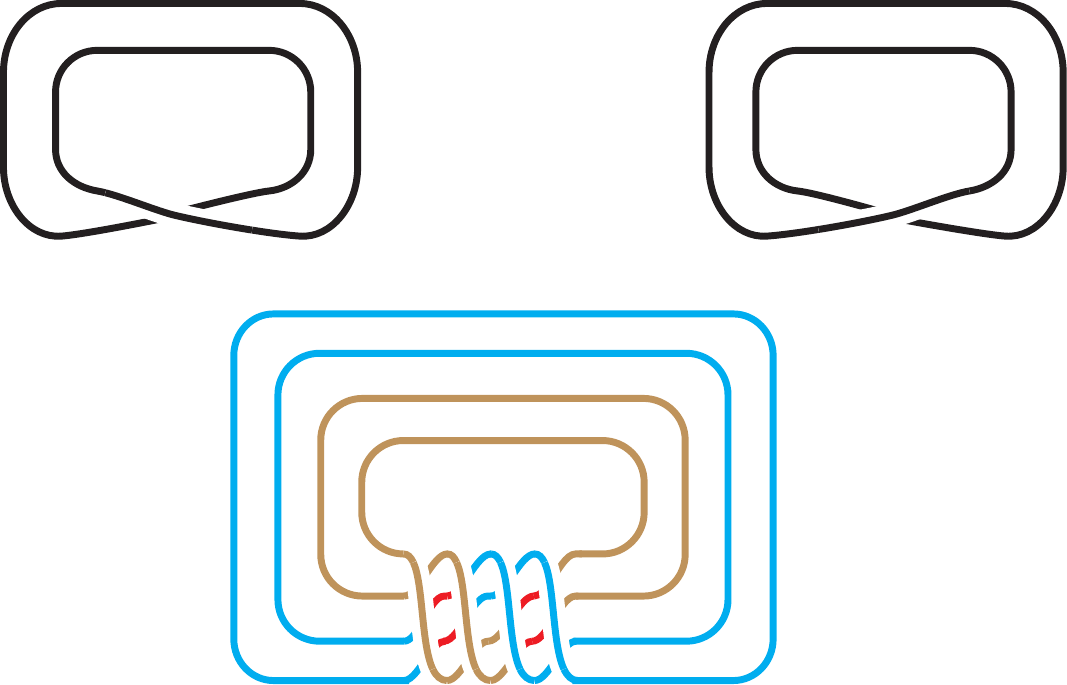}
\put(108,160){$(1\, 2)$}
\put(187,160){$(1\, 2)$}
\put(110,74){$(1\, 2)$}
\put(110,86){$(1\, 2)$}
\put(110,99){$(2\, 3)$}
\put(110,111){$(2\, 3)$}
\end{overpic}}
\caption{Upper left is the 2--fold cover branched along the transverse unknot with self-linking $-1$. The upper right is the 2--fold cover branched along the transverse unknot with self-linking $-3$.  The bottom figure describes a 3--fold simple branched cover of $S^3$ yielding $L(3,1)$. The colors on the strands represent elements of $S_3$. Blue represents $(2\, 3)$, brown represents $(1\, 2)$ and red represents $(1\, 3)$.}
\label{braidex}
\end{figure}
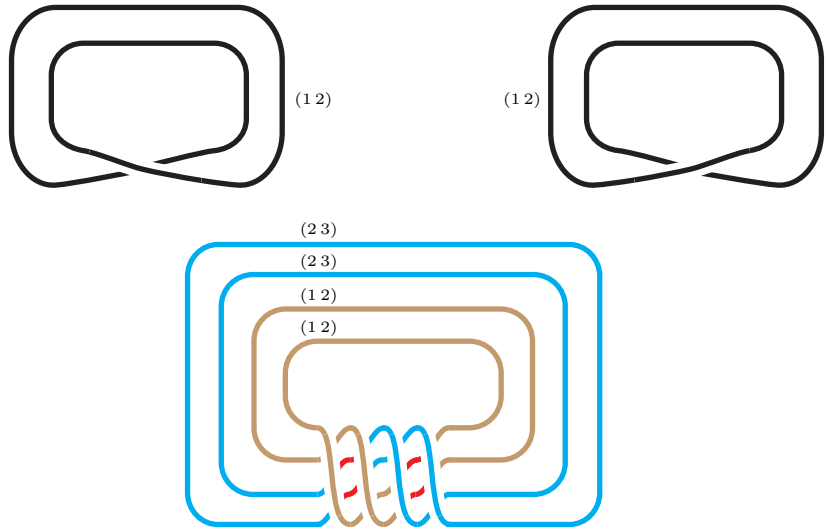
Moreover the monodromy describing the corresponding cover is also shown. We thus see that the 2--fold branched cover over the unknot shown in the figure simply yields $S^3$ with the standard contact structure (the given open book is simply a stabilization of the standard disk open book for $\xi_{std}$). If we take $b^{-1}$ then the branched cover will be $S^3$ with the overtwisted contact structure $\xi_1$ since the open book will be a negative stabilization of the standard open book for $S^3$ (see the end of Subsection~\ref{obd}). 
\end{example}
We record an observation from this example for future use.
\begin{lem}\label{lem:ot}
The  cover of $(S^3,\xi_{std})$ branched along the transverse unknot with self-linking $-3$ is the overtwisted contact structure $\xi_1$ on $S^3$ with $d_3(\xi_1)=1$. 
\end{lem}
\begin{proof}
From Equation~\eqref{slformula}  we know the transverse knot on the upper right of Figure~\ref{braidex} has self-linking $-3$. Moreover from \cite{Eliashberg93} it is known there is a unique such transverse knot. Now the computation in the above example yields the result. 
\end{proof}

\begin{example}\label{L31}
Consider now the 3--fold simple cover in Example~\ref{exsurfaces}. The diffeomorphism $b\colon D^2\to D^2$ given by a Dehn twist about a curve parallel to the boundary of $D^2$ lifts to the diffeomorphism $\widehat{b}\colon \Sigma\to \Sigma$ that is simply the composition of Dehn twists about curves parallel to each boundary component. It is well known, see \cite[Figure~3]{Etnyre04b}, that this open book describes the result of Legendrian surgery on the Legendrian unknot with Thurston-Bennequin invariant $-2$ and rotation number $\pm 1$. So the open book supports the lens space $L(3,1)$ with tight contact structure $\xi$ having $c_1(\xi)=\pm 1\in H^2(L(3,1))=\Z/3\Z$.
\end{example}

It is sometimes convenient to make the branch locus of a branched cover connected. We have the following contact version of the well-known result for topological branched covers. 
\begin{lem}[Casey 2013, \cite{Casey13}]\label{lem:connect}
Let $B$ be a transverse link in a contact 3--manifold $(Y,\xi)$ and $p\colon M\to Y$ be a simple cover branched along $B$ inducing the contact structure $\xi'$ on $M$. If part of a diagram for $B$ is as shown one side of  Figure~\ref{connect} then replacing that portion of $B$ with the other diagram shown in the figure will result in a new branched covering of $Y$ that still yields the same contact manifold $(M,\xi')$. \hfill \qed
\end{lem}
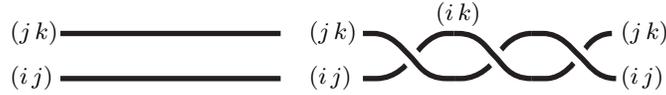
\begin{figure}[htb]{\small
\begin{overpic}
{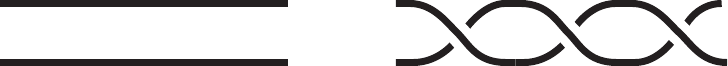}
\put(-19,-1){$(i\, j)$}
\put(-19, 16){$(j\, k)$}
\put(94,-1){$(i\, j)$}
\put(94,16){$(j\, k)$}
\put(142, 23){$(i\, k)$}
\put(212, -1){$(i\, j)$}
\put(212, 16){$(j\, k)$}
\end{overpic}}
\caption{Replacing the one diagram in the branch locus of a simple cover with the other does not change the manifold or contact structure described by the branched cover.}
\label{connect}
\end{figure}
The proof of the lemma follows easily by observing that the branched cover of the ball containing either  branched loci is simply a ball and the contact structure on it is tight. See \cite{Casey13} for details. 

We now make a useful observation about branched covers of contact 3--manifolds and stabilizations of transverse knots. 

\begin{prop}\label{stabilize}
Let $p\colon M\to Y$ be a simple branched covering between closed oriented 3--manifolds with branch locus $B\subset Y$. Let $\xi$ be a contact structure on $Y$ and $T$ be a transverse realization of $B$ in $(Y,\xi)$ and $T'$ the stabilization of $T$. The contact structure $\xi_{T'}$ on $M$ is obtained from the contact structure $\xi_{T}$ by connect summing with the overtwisted contact structure $(S^3,\xi_1)$. 

In particular, $\xi_{T'}$ is overtwisted, homotopic to $\xi_{T}$ over the 2--skeleton, and has $d_3$ invariant (when it is defined)
\[
d_3(\xi_{T'})=d_3(\xi_{T})+1.
\]
\end{prop}
\begin{proof}
The stabilization of $T$ can be done in a small neighborhood $N$ of a point on $T$ that only intersects $T$ in one arc and we can assume the contact structure on $N=D^2\times [-1,1]$ is given by $\ker(dz+r^2\, d\theta)$ and $T\cap N$ is $\{(0,0)\}\times [-1,1]$. Since $p\colon M\to Y$ is a simple cover, say an $n$--fold cover, the inverse image $p^{-1}(B)$ consists of $n-1$ balls $B_1,\ldots B_{n-1}$ and $p$ restricted to $B_2,\ldots B_{n-1}$ is a diffeomorphism. So the contact structure on each $B_i, i=2,\ldots, n-1$ is standard. The restriction of $p$ to $B_1$ is a 2--fold branched cover branched along the arc $T\cap N=\{0,0\}\times [-1,1]$. This is easily seen to be the standard contact structure on the ball too. When branching over $T'$ instead of $T$ there is no change in the contact structure on $M$ outside of the $B_i$ and on $B_i$ for $i=2,\ldots n-1$. 

We are left to determine the contact structure on $B_1$. To this end let $B$ be the 3--ball with its standard contact structure. We can glue $N$ and $B$ together to obtain $S^3$ with its standard contact structure. Moreover there is an arc $c$ in $B$ such that $T\cap N$ can be completed by $c$ to be the unknot with self-linking $-1$ and $T'\cap N$ can be completed by $c$ to the unknot with self-linking number $-3.$ Now the 2--fold cover of $(S^3,\xi_{std})$ branched along the unknot with self-linking number $-3$ is the overtwisted contact structure $\xi_1$ by Lemma~\ref{lem:ot}, and can be written as the union of the 2--fold cover of $N$ over $T'\cap N$ and the 2--fold cover of $B$ over $c$. Of course the contact structure on the first ball is the contact structure on $B_1$ induced from the covering map $p$ when branched along $T'$ and the contact structure on the second ball is standard. Thus we see that $\xi_{T'}$ is obtained from $\xi_T$ by connect summing with $(S^3,\xi_1)$ and the computation of $d_3(\xi_{T'})$ follows from the additivity of the $d_3$ invariant since $d_3(\xi_1)=1$.
\end{proof}

\section{Topological braiding in dimensions high and low}

In this section we explore braided embeddings and braided immersions in the first two subsections and prove the existence of certain braidings and give obstructions to others. We also see how to use braided embeddings to obstruct the branch locus of certain coverings of $S^n$ from being too simple. In the last subsection we generalize the notion of braiding and put it in a larger context. 
\subsection{Braided embeddings}
Given an $n$--manifold $Y$, a \dfn{braid about $Y$} is an embedding of an $n$--manifold $M$ into $Y\times D^2$ 
\[
e\colon M\to Y\times D^2
\]
such that $\pi\circ e\colon M\to Y$ is a branched covering map, where $\pi\colon Y\times D^2\to Y$ is projection onto the first factor. Recall that our standing assumption from Section~\ref{sec:bc} is that the branched locus of our covers will always be a submanifolds. Moreover we say a branched covering $p\colon M\to Y$ can be braided about $Y$ if there is a function $f\colon M\to D^2$ such that
\[
e\colon M\to Y\times D^2:x\mapsto (p(x), f(x))
\]
is an embedding (and hence exhibits $M$ as a braid about $Y$). If $Y$ is embedded in a $(n+2)$--manifold $W$ with trivial normal bundle and $M$ is braided about $Y$ then clearly $M$ also embeds in $W$ and this is called a \dfn{braided embedding of $M$ into $W$} (braided about $Y$). We will sometimes abuse terminology and refer to a braided embedding of $M$ as a realization of $M$ as a braid about $Y$ or as the embedding into some other ambient space $W$ as above.

Of course a given manifold $M$ can, potentially,  be braided about $Y$ in many different ways. Notice that when $n=1$ then the branched cover of a 1--manifold is an actual covering map (since the branch locus must be co-dimension 2). Thus a braid about $S^1$ is an embedding of $S^1$ into $S^1\times D^2$ that is transverse to $\{p\}\times D^2$ for all $p\in D^2$. That is our notion of braiding coincides with the ordinary notion of a closed braid in dimension 3. 

When the branched cover corresponding to a braiding of $M$ about $Y$ has a property, such as being simple or cyclic, we will use the same adjective to describe the braiding, for example we will refer to a ``simple braiding" when $M$ is braided about $Y$ so that the corresponding branched covering is simple. 

It is interesting to consider when a given branched covering map $p\colon M\to Y$ can be realized by a braiding of $M$ about $Y$. This question has been addressed in \cite{CarterKamada12, CarterKamada15b, CarterKamada15} and in particular in \cite{CarterKamada15b} an example was given showing that not all branched covers can be so realized. In Example~\ref{nobraid} below we use contact geometry to give an infinite family (and a recipe for constructing more infinite families) of examples of branched covers that cannot be realized as a braid about $S^3$.  But for now we start by observing there are branched covers that can always be realized as braidings.  

\begin{thm}[Hilden 1978, \cite{Hilden78}]\label{cyclicbraid}
Let $p\colon M\to Y$ be a cyclic branched cover between closed oriented $n$--manifolds with branch locus $B\subset Y$ a closed, orientable, and null-homologous submanifold. Then there is a function $f\colon M\to D^2$ so that 
\[
e\colon M\to Y\times D^2:x\mapsto (p'(x), f(x))
\]
exhibits $M$ as a braid about $Y$, where $p'$ is homotopic to $p$ through cyclic branched covers. 
\end{thm}
This theorem is due to Hilden \cite{Hilden78} but a proof is given below for completeness. Unaware of Hilden's paper, the theorem was also rediscovered for $2$--fold branched covers in dimension $2$, $3$ and $4$ in \cite{CarterKamada15b}.

For many $Y$ and null-homologous submanifold $B$ there can be more than one branched cover that is a cyclic cover in the complement of $B$. But when we say ``cyclic branched cover" we mean the branched cover of $Y$ that unwinds each meridian to $B$ according to the fold of the cover. (More precisely consider the map $\pi_1(Y-B)\to \Z$ obtained by abelianizing followed by the map $H_1(Y-B;\Z)\to \Z$ induced by intersecting with some 
chosen connected oriented Seifert hypersurface for $B$. 
 Then the inverse image of $n\Z$ gives the subgroup defining the cyclic cover. This subgroup can depend on the Seifert hypersurface, but  the theorem is true for the cyclic cover corresponding to any choice of hypersurface.) 
\begin{remark}
The orientability hypothesis for $B$ is essential as demonstrated in Example~\ref{trivialnb} below. In particular, this example shows that not all branched covers can be braided (or even immersed braided) about $S^n$. 
\end{remark}

\begin{proof}
Let $S$ be a Seifert hypersurface for $B$, that is a co-dimension $1$ connected submanifold $S$ of $Y$ such that $B=\partial S$, corresponding to a given cyclic branched covering as discussed above. 
We will define a smooth function $h\colon Y\to \C$ such that $0$ is a regular value, $h^{-1}(0)=B$, and for any loop $\gamma$ in the complement of $B$ its algebraic intersection with $S$ is given by winding of $h\circ \gamma$ about $0\in \C$. 
Given the function $h$ let 
\[
{X}=\{(x,z)\in Y\times \C: z^n=h(x)\}.
\]
It is clear that the map $p'\colon {X}\to Y:(x,z)\mapsto x$ is the $n$--fold cyclic cover of $Y$ branched along $B$ (indeed it is clearly an $n$--fold covering map in the complement of the branch locus and unwraps each meridian as desired) that is homotopic to $p$. Thus $X$ is diffeomorphic to $M$ and restricting the projection $Y\times \C\to \C$ to ${X}$ will give us the function $f$ claimed in the theorem. 

We are left to construct $h$. Use $S$ to provide a framing for the normal bundle of $B$ and use this framing to identify a tubular neighborhood of $B$ with $N=B\times D^2$ where we are thinking of $D^2$ as the unit disk in $\C$ and $S\cap N$ agrees with $B$ times the positive real axis. Define $h\colon B\times D^2\to D^2$ by projection and extend it to all of $Y$ as follows. Identify a neighborhood of $S\cap (Y-N)$ with $N'=S\times (-\epsilon, \epsilon)$ for some small $\epsilon>0$ and define $h$ on $S\times (-\epsilon, \epsilon)$ by $h(x,t)=e^{it}$. Notice that we have $h$ defined on $\partial \overline{(Y-N\cup N')}$ so that the image is contained in $\partial D^2$ minus a neighborhood of $1$. That is the image is contained in an interval and hence we can extend $h$ over $Y-(N\cup N')$ such that $h\not=0$ there. We can now approximate $h$ by a smooth function relative to $N$. As this approximation can be made arbitrarily small we can guarantee that 0 is still a regular value and $B=h^{-1}(0)$. 
\end{proof}

While it is not true that all branched covers of a 3--manifold over $S^3$ can be realized by a braiding, see Example~\ref{nobraid}, Hilden, Lozano and Montesinos \cite{HildenLozanoMontesinos83}, see also \cite[Section~5.3]{Montesinos85}, proved that all 3--manifolds can be braided even if not via a specific branched covering.

\begin{thm}[Hilden, Lozano and Montesinos 1983, \cite{HildenLozanoMontesinos83}]\label{HLM}
Every closed oriented 3--manifold $M$ can be braided about $S^3$ where the corresponding branched cover is a simple 3-fold branched cover.
\end{thm}

If $Y$ is embedded in an $(n+2)$--manifold $X$ with trivial normal bundle, then we denote a tubular neighborhood of $Y$ in $X$ by $N=Y\times D^2$. We say an embedding of $M$ in $X$ can be \dfn{braided about $Y$} if the embedding can be isotoped to lie in $N$ such that it is a braid about $Y$. We recall the following known results.
\begin{thm}[Alexander 1923, \cite{Alexander23} for $n=1$ and Kamada 1994, \cite{Kamada94} for $n=2$]
Let $S^n\subset S^{n+2}$ be the standard embedding of $S^{n}$ in $S^{n+2}$. If $n=1$ or $n=2$ then any embedding of an $n$--manifold in $S^{n+2}$ can be braided about $S^n$. If $n=2$ then we can take the associated branched covering to be simple.  
\end{thm}

It is not currently known if all embeddings of a 3--manifold into $S^5$ can be braided about the standardly embedded $S^3$. While it is conjectured that such braidings do exist, we demonstrate that there is at least an infinite family of isotopy classes of embeddings of $S^3$ into $S^5$ that can be so realized. 
\begin{example}\label{differentembeddings}
We will construct a family of embeddings as ``open book embeddings". That is given open book decompositions $(\Sigma, \phi)$ and $(\Sigma',\phi')$ if we have a proper embedding $e:\Sigma\to \Sigma'$ such that $\phi'$ fixes the image of $e$ and when restricted to it is conjugate to $\phi$ via $e$, then we may clearly use $e$ to embed the mapping torus $T_\phi$ into $T_{\phi'}$ and extend this to an embedding of $M_{(\Sigma,\phi)}$ into $M_{(\Sigma',\phi')}$ by defining it to be $\partial \Sigma\times D^2\to \partial \Sigma'\times D^2:(p,x)\mapsto(e(p),x)$ on a neighborhood of the binding.

To construct our embeddings we will consider the open books $(D^2,id_{D^2})$ and $(D^4,id_{D^4})$ for $S^3$ and $S^5$, respectively, where we are thinking of $D^4$ as $D^2_1\times D^2_2$.  These open books induce the decompositions  $S^3= (S^1\times D^2)\cup (D^2\times S^1)$ and $S^5= (S^1\times D^4)\cup (D^2\times S^3)$. (In both cases the second factor is the neighborhood of the binding and all disks are unit disks in $\C$.) Embedding the $D^2$ page of $S^3$ into $D^4=D^2_1\times D^2_2$ as $D^2_1\times \{0\}$ we obtain an open book embedding as discussed above. Specifically, thinking of $S^5$ as the unit sphere in $\R^6$, the embedding of $S^3$ into $S^5$ has image $S^5$ intersected with the $\R^4\subset \R^6$. More explicitly, we can embed $S^3\times D^2$ as a neighborhood of this standardly embedded $S^3$ in $S^5$ as follows:
\[
(S^1\times D^2\times D^2)\to (S^1\times D_1^2\times D^2_2):(\theta, z, w)\mapsto (\theta, z, \frac 12 w)
\]
and
\[
(D^2\times S^1\times D^2)\to D^2\times ({\partial D_1^2\times D_2^2}): {(z',\theta', w )\mapsto(z',\theta', \frac 12 w)}
\]
where we are thinking of $S^3=\partial D^4$ as $({\partial D_1^2}\times D^2_2)\cup (D_1^2\times \partial D_2^2)$. Denote this embedding by $e$.

We recall that a quasi-positive braid is simply a braid that is written as a product of conjugates of the standard Artin generators of the braid group. We refer the reader to \cite{EtnyreVanHorn-Morris??, Rudolph83a} for details on quasi-positive braids. In particular, given an $n$ braid written as a product of $k$ conjugates of generators one can construct a ribbon immersion of a surface $\Sigma$ into $S^3$ whose boundary is the closure of the braid and has Euler characteristic $n-k$. In \cite[Section~2]{Rudolph83a} it was shown that there is also an embedding of $\Sigma$ into $D^2_1\times D^2_2$ so that the projection to $D^2_1$ restricted to $\Sigma$ is an $n$-fold branched covering map, with $k$ simple branch points. 

Now consider the quasi-positive braid $\sigma_1(\sigma_2^{2n+1} \sigma_1\sigma_2^{-(2n+1)})$, where $\sigma_1$, $\sigma_2$ are the Artin generators of the braid group with three strands $B_3$. The corresponding surface $\Sigma_n$ is a disk with a ``braided embedding" $f_n$ into $D^2_1\times D^2_2$, and one can assume that $f_n$ has image in $D^2_1\times \frac 12 D^2_2$. We are now ready to construct our braided embeddings of $S^3$ into $S^5$. Thinking of $S^3$ as $(S^1\times D^2)\cup (D^2\times S^1)$ we define the embedding $e_n$ as follows:
\[
(S^1\times D^2)\to (S^1\times D_1^2\times D^2_2): (\theta, z)\mapsto (\theta, f_n(z))
\]
and
\[
(D^2\times S^1)\to D^2\times ({\partial D_1^2\times D_2^2}): {(z',\theta')\mapsto (z',f_n|_{\partial \Sigma_n}(\theta'))}.
\]
Notice that $e_n$ is clearly a braided embedding about $e(S^3\times\{0\})$. 
It is also easy to see that the branched locus is a trivial link with two components and the branched cover is 3-fold and simple. 

To show our embeddings are non-isotopic we compute the fundamental groups of their complements. To this end we note it is not too hard to compute $\pi_1((D^2_1\times D^2_2)-f_n(D^2))$, for example an algorithm (similar to the Wirtinger presentation of the fundamental group of a knot complement) is given in \cite[Section~4]{Rudolph83a} from which one easy sees that the group has presentation
\[
\left\langle x_1,x_2,x_3| {x_1x_2^{-1}}, x_1^{-1}\left((x_2x_3)^{n+1}x_3(x_3^{-1}x_2^{-1})^{n+1}\right) \right\rangle
\]
and hence also has presentation 
\[
\langle x_2,x_3|  x_2^{-1}((x_2x_3)^{n+1}x_3(x_3^{-1}x_2^{-1})^{n+1}) \rangle.
\]

Notice that the complement of $e_n(S^3)$ in $S^5$ can be written as the union of two parts just as the embedding was defined in two parts. The first part is the complement of the image of $f_n$ times $S^1$ and the second part is the complement of the closure of the braid in $S^3$ times $D^2$. A simple application of {van-Kampen's} Theorem thus gives that $\pi_1(S^5-e_n(S^3))$ is isomorphic to $\pi_1((D^2_1\times D^2_2)-f_n(D^2))$. One may check that these groups are isomorphic to the fundamental group of the complements of the $(2,2n+1)$ torus knots. As it is well known that these groups are non-isomorphic it is clear that all the braided embeddings $e_n$ are non-isotopic. 
 \end{example}

\begin{remark}
Notice that in the last example we could have used any quasi-positive sliced knot to construct a braided embedding of $S^3$ into $S^5$. This would lead to many other non-isotopic braided embeddings. 
\end{remark}

\subsection{Braided immersions}\label{sec:immersions}
We can easily define an immersed version of braiding. Given an $n$--manifold $Y$, an \dfn{immersed braid about $Y$} is an immersion of an $n$--manifold $M$ into $Y\times D^2$ 
\[
i\colon M\to Y\times D^2
\]
such that $\pi\circ i\colon M\to Y$ is a branched covering map, where $\pi\colon Y\times D^2\to Y$ is projection onto the first factor. Moreover we say a branched covering $p\colon M\to Y$ can be realized by an immersed braid about $Y$ if there is a function $f\colon M\to D^2$ such that
\[
i\colon M\to Y\times D^2:x\mapsto (p(x), f(x))
\]
is an immersion (and hence exhibits $M$ as an immersed braid about $Y$). 
\begin{thm}\label{immerse}
Let $p\colon M\to Y$ be any branched cover between closed oriented $n$--manifolds with branch locus $\widetilde{B}\subset M$ a submanifold having trivial normal bundle. Then there is a function $f\colon M\to D^2$ so that 
\[
i\colon M\to Y\times D^2:x\mapsto (p(x), f(x))
\]
exhibits $M$ as an immersed braid about $Y$. 
\end{thm}
We note that in \cite{CarterKamada15b} this theorem was also proven for the case of simple 3--fold covers in dimension 1, 2 and 3 when $Y$ is a sphere. 
\begin{remark}
The hypothesis on the normal bundle of $\widetilde{B}$ could be replaced by the stronger hypothesis that the normal bundle to branch locus $B$ in $Y$ is trivial.  
\end{remark}
\begin{proof}
Given the branched cover $p\colon M\to Y$ with branch locus $\widetilde{B}\subset M$ having trivial normal bundle, let $N=\widetilde{B}\times D^2$ be a small tubular neighborhood of $\widetilde{B}$ in $M$. Define $f\colon M\to D^2$ on $N$ by projection to the second factor and then extend it to the rest of $M$ arbitrarily. 

Clearly $di=dp\oplus df\colon TM\to TY\oplus TD^2$. At all points $x\in M-\widetilde{B}$ we know $dp_x$ has rank $n$ and so $di_x$ does too. Moreover, at points $x\in \widetilde{B}$ we know $dp_x$ has rank $n-2$ and $df_x$ has rank $2$ on the kernel of $dp_x$. Thus $di_x$ has rank $n$ on all of $M$ and hence $i\colon M\to Y\times D^2$ is an immersion. 
\end{proof}

The example below shows that we can use Theorem~\ref{immerse} to give obstructions to the possible branch locus for a branched cover $M\to Y$. It also shows the necessity of the hypothesis on the branched locus in both Theorems~\ref{cyclicbraid} and~\ref{immerse}.

\begin{example}\label{trivialnb}
It is known that $\C P^n$ does not immerse in $\R^{2n+2}$, for $n>1$. To see this, recall Hirsch \cite{Hirsch59} shows that an $n$--manifold $M$ immerses in $\R^{n+2}$ if and only if there is a 2--dimensional bundle $L$ over $M$ such that $TM\oplus L$ is the trivial bundle. Now if $p(E)$ is the total Pontryagin class of a bundle $E$, then clearly $p(TM\oplus L)=1$. Moreover, the ``Whitney sum formula'', \cite[Theorem~15.3]{MilnorStasheff74}, says that $p(TM\oplus L)=p(TM)\cup p(L)$ modulo elements of order 2. We also know, see \cite[Example~15.6]{MilnorStasheff74}, that the total Pontryagin class of $\C P^n$ is $(1+a^2)^{n+1}$ where $a$ is the generator of $H^2(\C P^n)$ and since $L$ is a 2--dimensional bundle $p(L)=1$. Thus we see that if $\C P^n$ immerses in $\R^{2n+2}$, then $(1+a^2)^{n+1}=1$. But notice that the coefficient on $a^2$ is $n+1$. So if $a^2\not=0$ then $\C P^n$ has no such immersion. 

If $\C P^n$  could be realized as a branched cover over $S^{2n}$ with smooth branch locus having trivial normal bundle then the above theorem would immerse it into $S^{2n}\times D^2$ and hence it could be immersed in $\R^{2n+2}$.  Also notice that if the branched set was orientable then it would have trivial normal bundle, see \cite[Theorem~VIII.2]{Kirby89}. Thus we see that the branch locus for any branched cover of $\C P^n$ over $S^{2n}$ must either be non-embedded or non-orientable. This result was already known in the case of $n=2$ using an Euler characteristic argument \cite{Piergallini95}, but to the best of our knowledge it was not known in higher dimensions. Moreover, it demonstrates how to use braided embeddings/immersions to obtain information about the possible branched loci.

We also note it is well-known that $\C P^2$ is a 2-fold (cyclic) branched cover over an $\R P^2$ embedded in $S^4$ with normal Euler number $2$, \cite{Massey73}. This shows the hypothesis on the branch locus in Theorems~\ref{cyclicbraid} and~\ref{immerse} is essential. 
\end{example}

\begin{example}\label{trivialnb2}
As another example we consider $\R P^n$. Using the total Steifel-Whitney classes it is well known that if $\R P^{2^k}$ immerses in $\R^{2^{k}+c}$ then $c\geq 2^k-1$, see \cite[Theorem~4.8]{MilnorStasheff74}. Thus since $\R P^k$ embeds in $\R P^l$ for $k\leq l$, we see that if $\R P^{2^k-d}$, where $0\leq d<2^{k-1}$, immerses in $\R^{2^{k}-d+2}$ then $0< d\leq 3$. Thus the only possible $n=2^k-d$ where $\R P^n$ has a co-dimension 2 immersion is when $d=1,2,$ or $3$. 
When $k>3$, these cases are ruled out by \cite{Atiyah62, DavisPC}. Thus $\R P^n$ has no co-dimension 2 immersion in Euclidean space for $n>7$. It is known that for $n\leq 7$ there are co-dimension 2 (and sometimes even 1) immersions of $\R P^n$. 

As we argued in the previous example these results imply that $\R P^n$ cannot be realized as a cover of $S^n$, branched over an embedded orientable submanifold, if $n>7$. This recovers most of a result of Little, \cite{Little84}, that says if $\R P^n$ is a cover of $S^n$ branched over a locally flat oriented submanifold then $n=1,3,$ or $7$.  (It is still unknown if $\R P^7$ can be realized as such a branched cover.)
\end{example}

\subsection{Higher co-dimension braids}
Once can consider braids, and immersed braids in higher co-dimension. We set this up as part of a more general interesting question. Given two (possibly singular) bundles $p\colon M\to Y$ and $\pi\colon E\to Y$ one can ask the following question.
\begin{question}
When does there exist an embedding  (or immersion) $e\colon M\to E$ such that $p=\pi\circ e$? 
\end{question}
Said more colloquially, ``When can one embed one bundle in another?" By ``possibly singular" bundles we mean for example that one of the bundles could be, say, a branched cover and the other could be a Lefschetz fibration or other such object. Notice when $\pi\colon E\to Y$ is an honest bundle or Lefschetz fibration then the existence of a bundle embedding of a covering space is the same as the much studied question concerning the existence of a multi-section. 

Similarly if $p\colon M\to Y$ is a branched covering and $\pi\colon Y\times D^k\to Y$ is projection onto the first factor, then an embedding $e\colon M\to Y\times D^k$ for which $\pi\circ e=p$, will be called a \dfn{co-dimension $k$ braiding of $M$ about $Y$} and similarly for immersions. 

We note that if $p\colon M\to Y$ can be realized as a co-dimension $k$ braid about $Y$ then it can be realized as a co-dimension $l$ braid about $Y$ for all $l\geq k$. So Theorems~\ref{cyclicbraid} and~\ref{immerse} give conditions guaranteeing braiding in all co-dimension above $1$. 

We do not have much to say about the general braiding problem, but do ask a couple of questions.
\begin{question}
Can you use higher co-dimensional braiding to give restrictions on the branched set for any branched cover of a given $n$--manifold over $S^n$? For example as was done in Example~\ref{trivialnb} and Example~\ref{trivialnb2} using co-dimension 2 braids. 
\end{question}
\begin{question}
What can generalized braiding say about the smallest dimensional Euclidean space into which you can embed a given manifold? Can such an optimal embedding always be obtained through braiding?
\end{question}

\section{Contact embeddings via braids}
In this section we show how to use braided embeddings to produce contact structures on the ``braid" manifold. We will also use this connection to contact geometry to construct branched covers over $S^3$ that cannot be realized as braids. 

\begin{thm}\label{MainContactBraid}
Let $M$ and $Y$ be a closed oriented (2n+1)--manifolds and 
\[
e\colon M\to Y\times D^2:x\mapsto (p(x), f(x))
\]
a braiding of $M$ about $Y$ such that the branched covering $p\colon M\to Y$ whose branch locus $B\subset Y$ is an orientable submanifold that is not multiply ramified (that is at most one component of the pre-image of each component of $B$ is ramified). There is an orientation on $B$ such that given a contact structure $\xi=\ker \alpha$ on $Y$ and any contact structure $\xi'$ induced on $M$ by the covering $p$ branched along any realization of $B$ as a (positive) transverse contact submanifold, then $e$ may be isotoped so that it is a contact embedding of $(M,\xi')$ into $(Y\times D^2, \ker (\alpha+r^2\, d\theta))$ and moreover the image of $e$ can be assumed to lie in an arbitrarily small neighborhood of $Y\times\{(0,0)\}$. 
\end{thm}

A key idea behind the theorem is that away from the branch locus as we scale $f(x)$ by a small $\epsilon$ the tangent space to the image of $e$ becomes arbitrarily close to the tangent space of $Y\times\{(0,0)\}$. Since the contact condition is open it is clear that $\alpha+r^2\, d\theta$ will restrict to a contact form there. The proof of the theorem, given below, then follows by paying close attention to a neighborhood of the branch locus (for which we have a precise local model) and noting that the induced contact form on $M$ agrees with one giving the contact structure on the branched cover.

\begin{remark}
We note that to apply this theorem we must be able to realize $B$ as a contact submanifold of some contact structure on $Y$. Since $B$ is co-dimension~2 it is not clear if this can always be done except in the case when $B$ is 1~dimensional. 

In particular, this theorem gives a way to try and isotope embeddings of 3--manifolds in $S^5$ to be transverse contact embeddings. We can now prove Theorem~\ref{braidedembedimplycont} which says that if an embedding $M\to S^5$ can be isotoped to be a braided embedding about the standard $S^3$ in $S^5$ then it can be isotoped to be transverse contact embedding.
\end{remark}
\begin{proof}[Proof of Theorem~\ref{braidedembedimplycont}]
The standard embedding of $S^3$ in $S^5$ gives a contact embedding of the standard contact structures. Thus by Proposition~\ref{nbhdprop}, $S^3$ has a neighborhood $S^3\times D^2$ with contact structure given by $\ker(\alpha_{std}+r^2\, d\theta)$, where $\alpha_{std}$ is a contact form for the standard contact structure on $S^3$. 

We now show how one can isotope the embedding so that branched covering corresponding to the embedding has a branch locus which is not multiply ramified so that we can apply Theorem~\ref{MainContactBraid}. Let $\widetilde{B}_1$, $\widetilde{B}_2$ be distinct components of the ramified set in $M$ lying above a component $B$ of the branch locus of $p$.  There exist neighborhoods $\widetilde{N}$ of $\widetilde{B}_1$ and $N$ of $B$ such that $\widetilde{N}$ does not contain ramified points other than $\widetilde{B}_1$, $N\cong S^1\times D^2$ does not contain branch points other than $B$, and $B$ is identified with $S^1\times \{0\}$.
Let $\psi_{t}$, $t\in[0,1]$, be an isotopy generated by a vector field supported in $N$, tangent to the $D^2$-factors of $N$, and non-zero along $B$. We now define the map $p_t\colon M\to Y$ to be $p$ on ${M\setminus\widetilde{N}}$ and $\psi_t \circ p$ on $\widetilde{N}$. This is clearly an isotopy of $p$ and and hence induces an isotopy of $e_t=(p_t,f)\colon M\to S^3\times D^2$ of $e$ through braided embeddings. Notice that for $t>0$ a copy $B'$ of $B$ is added to the branch locus (specifically $B'$ is the image of $B$ under $\psi_t$). By construction the branching above $B'$ is simple and the number of ramified components above $B$ is reduced by one. 
By repeating this process finitely many times, we can isotope the given $e$ to a braided embedding whose branch locus is not multiply ramified. Since any link in $S^3$ can be isotoped to be transverse to the standard contact structure on $S^3$ we can clearly isotope the given embedding to satisfy the hypothesis of Theorem~\ref{MainContactBraid} and thus the theorem gives the desired isotopy. 
\end{proof}

\begin{example}\label{nobraid}
In this example we construct infinitely many branched covers of $S^3$ that cannot be realized as a braid about $S^3$. 

In Example~\ref{L31} we saw that $L(3,1)$ is the simple 3--fold cover of $S^3$ branched along the lower diagram in Figure~\ref{braidex} and that the contact structure $\xi$ induced on $L(3,1)$ from this cover has $c_1(\xi)=\pm1$. Using Lemma~\ref{lem:connect} three times we can change the branch locus of this cover to $B$ shown in Figure~\ref{fig:knotlocus}. 
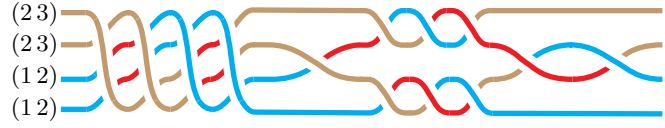
\begin{figure}[htb]{\small
\begin{overpic}
{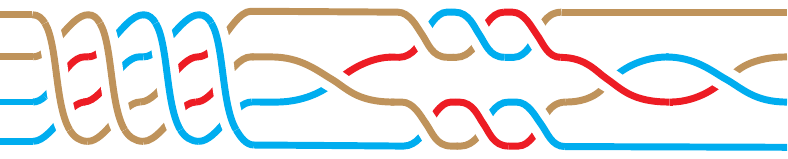}
\put(-19,3){$(1\, 2)$}
\put(-19, 15){$(1\, 2)$}
\put(-19,27){$(2\, 3)$}
\put(-19,39){$(2\, 3)$}
\end{overpic}}
\caption{The closure of this braid represents a 3--fold simple branched cover of $S^3$ yielding $L(3,1)$. The colors on the strands represent elements of $S_3$. Blue represents $(1\, 2)$, brown represents $(2\, 3)$ and red represents $(1\, 3)$.}
\label{fig:knotlocus}
\end{figure}
Notice that $B$ is a knot and that either orientation on $B$ defines the same knot. Let $p\colon L(3,1)\to S^3$ be the branched cover with branch locus $B$ described by the branching data in Figure~\ref{fig:knotlocus}. If this branched cover could be braided about $S^3$ then $(L(3,1),\xi)$ contact embeds in $S^3\times D^2$ with contact structure given by $\ker(\alpha_{std}+r^2\, d\theta)$, where $\alpha_{std}$ is a contact form for the standard contact structure on $S^3$, and thus in $(S^5, \xi_{std})$. But this contradicts Theorem~\ref{thm:obstruct}. Thus there is no such embedding. (Note we needed to have a connected branch locus so that we only had to be concerned with the orientation on a knot and we constructed a knot which is isotopic to its reverse.) 

We note that one may easily write down an infinite family of branched covers that do not embed as follows. Arguing as in Example~\ref{L31} one takes the $2n$ braid, with $n>1$, with one full twist and labels the first two strands on the left by $(1\, 2)$, the next two by $(2\, 3)$ and so on until the last two are labeled by $(n\, n+1)$. Taking the closure of this braid and extending the labeling by the ``Wirtinger relations" at the crossings will describe an $(n+1)$--fold simple branched cover of $S^3$ branched along the given link. The cover will be $L(n+1, 1)$ with induced contact structure $\xi_n$ having $c_1(\xi_n)= (n-1) g$ where $g$ is a generator of $H^2(L(n+1,1))=\Z/(n+1)\Z$. Thus turning the branch locus into a reversible knot as above gives a branched cover that cannot be braided about $S^3$. 

There are many other infinite families that can similarly be constructed. These examples should be compared with the example found in \cite{CarterKamada15b}.
\end{example}

To prove Theorem~\ref{MainContactBraid} we need the following technical lemma. 

\begin{lem}\label{mainlem}
Let $M$ and $Y$ be closed oriented (2n+1)--manifolds and 
\[
e\colon M\to Y\times D^2:x\mapsto (p(x), f(x))
\]
a braiding of $M$ about $Y$. Denote the branch locus of $p$ by $B\subset Y$. Given a contact structure $\xi=\ker \alpha$ on $Y$ in which  $B$ is a transverse contact submanifold, then let $\xi'$ be the contact structure on $M$ induced by the branched cover $p$. 

Let $\widetilde{B}'$ be the subset of $ \widetilde{B}=p^{-1}(B)$ at which $p$ is ramified. If for all $x\in\widetilde{B}'$ the map $df_x\colon T_xM\to T_{f(x)}D^2$ is orientation preserving when restricted to the fiber of the normal bundle $\nu_x(\widetilde{B}')$, where $\nu_x(\widetilde{B}')$ is oriented by the orientation on $\widetilde{B}'$ and $M$, and $D^2$ with polar coordinates $(r,\theta)$ is oriented by $rdr\wedge d\theta$, then for all small $R>0$ the embedding
\[
e_R\colon M\to Y\times D^2:x\mapsto (p(x), Rf(x))
\] 
is a contact embedding from $(M,\xi')$ to $(Y\times D^2,\ker (\alpha+r^2\, d\theta))$.
\end{lem}

We note that an immediate corollary of this lemma and the proof of Theorem~\ref{cyclicbraid} is the following result that will be used below. 
\begin{cor}\label{contctforcyclic}
Let $(Y,\xi)$ be a contact $(2n+1)$--manifold and $B$ a co-dimension 2 contact submanifold that is null-homologous and has trivial normal bundle. Let $(M,\xi')$ be the contact structure obtained from $(Y,\xi)$ by an $n$-fold cyclic branched cover branched along $B$. Then there is a braided contact embedding of $(M,\xi')$  into  $(Y\times D^2, \ker(\alpha+r^2\, d\theta))$, where $\alpha$ is a contact form for $\xi$.\hfill \qed
\end{cor}

We first establish Theorem~\ref{MainContactBraid} given Lemma~\ref{mainlem} and then prove the lemma. 
\begin{proof}[Proof of Theorem~\ref{MainContactBraid}]
Given the embedding $e\colon M\to Y\times D^2$ as in the statement of the theorem, let $\widetilde{B}=p^{-1}(B)$ and $\widetilde{B}'$ be the subset of $\widetilde{B}$ on which $p$ is actually ramified. Recall by hypothesis $p$ maps each component of $\widetilde{B}'$ to a different component of $B$. 

At any point $x\in \widetilde{B}'$ notice that $df_x$ gives an isomorphism from the fiber of the normal bundle $\nu_x(\widetilde{B}')$ to $T_{f(x)}D^2$ since the map $de_x\colon T_xM\to T_{e(x)}(Y\times D^2)$ has rank (2n+1), but $dp_x\colon T_xM\to T_{p(x)}Y$ has only rank $2n-1$. Thus at each point of $\widetilde{B}'$ there is an induced orientation on the fibers of $\nu(\widetilde{B}')$ and this orients each component of $\widetilde{B}'$, which in turn induce an orientation on $B$ via $p$. 

Now if $B$ can be isotoped to a positive transverse contact submanifold then there is an ambient isotopy $\phi_t\colon Y\to Y, t\in[0,1]$ that realizes this isotopy. Thus there is a diffeomorphism of $Y\times D^2$ that takes $e$ to $e'\colon M\to Y\times D^2:x\mapsto (\phi_1\circ p(x), f(x))$. And $e'$ realizes $M$ as braided about $Y$ and the corresponding branched set is the transverse realization of $B$. The theorem now follows from Lemma~\ref{mainlem}.
\end{proof}

\begin{proof}[Proof of Lemma~\ref{mainlem}]
Let $\beta_R=e_R^*(\alpha+r^2\, d\theta)= p^*\alpha + Rf^*(r^2\, d\theta)$. The contact condition concerns the form $\beta_R\wedge(d\beta_R)^n$ which is equal to
\[
p^*(\alpha\wedge (d\alpha)^n)+ R\left(p^*((d\alpha)^n)\wedge f^*(r^2\, d\theta)\right) +2nR\left(p^*(\alpha\wedge(d\alpha)^{n-1})\wedge f^*(r\, dr\wedge d\theta)\right).
\]
Away from $\widetilde{B}'$, $p$ is a covering map so the first term is a positive multiple of the volume form. Thus for $R$ sufficiently small $\beta_R$ is a contact form on the complement of a neighborhood of $\widetilde{B}'$. On the branch locus $\widetilde{B}'$ recall that $p$ has rank $2n-1$ and more specifically is a covering map when restricted to $\widetilde{B}'$ and has 0 derivative in the normal directions to $\widetilde{B}'$. Thus the first two terms in the expression for $\beta_R\wedge(d\beta_R)^n$ above are zero and the last term is a positive multiple of the volume form for $M$. This is clear by the hypothesis on $f$ in the lemma and the fact that $p^*(\alpha\wedge(d\alpha)^{n-1})$ is positive volume form on $\widetilde{B}'$ and $f^*(r\, dr\wedge d\theta)$ is an area from on the fiber to the normal bundle $\nu(\widetilde{B}')$. Moreover it is clear from the form of $\beta_R$ that it gives the contact structure $\xi'$ coming from the cover $p\colon M\to Y$ branched along $B$. 
\end{proof}

\section{Contact embeddings of 3--manifolds in $(S^5,\xi_{std})$}
We begin with a simple observation. 
\begin{prop}\label{oneembed}
Any closed oriented 3--manifold has some, possibly overtwisted, contact structure that embeds in $(S^5,\xi_{std})$.
\end{prop}

\begin{proof}
Given a closed oriented 3--manifold $M$, Theorem~\ref{HLM} tells us that there is a braided embedding
\[
e\colon M\to S^3\times D^2
\]
such that the corresponding branched covering is a simple 3--fold branched cover. Thus $e$ satisfies the hypothesis of Theorem~\ref{MainContactBraid} and since the branch locus can be isotoped to be a transverse link in $(S^3,\xi_{std})$ the contact structure $\xi'$ induced on $M$ by this branched cover contact embeds in $(S^3\times D^2, \ker(\alpha+r^2\, d\theta)),$ where $\alpha$ is a contact form for the standard contact structure on $S^3$. 

Now of course the standard embedding of $S^3$ into $S^5$ is also an embedding of the standard contact structures. Hence, by Proposition~\ref{nbhdprop}, $S^3$ has a neighborhood $S^3\times D^2$ in $S^5$ on which the contact structure is given by  $\ker(\alpha+r^2\, d\theta)$. Since the contact embedding from Theorem~\ref{MainContactBraid} can be arranged to be arbitrarily close to $S^3\times\{(0,0)\}$ we see that $M$ has a contact embedding into $(S^5,\xi_{std})$ that is arbitrarily close to the embedding of $S^3$. 
\end{proof}

\begin{remark}
We note that the braided embeddings constructed above can also also be made into open book embeddings as discussed in Examples~\ref{differentembeddings}. Indeed, to see this we first note that the standard embedding of $S^3$ in $S^5$ is such an embedding. Now given a contact embedding of a 3--manifold $M$ constructed as a braid about the standardly embedded $S^3$ as in the proof of Proposition~\ref{oneembed}, 
by applying Alexander theorem for (transverse) links to its branch locus in the standardly embedded $S^3$, we can isotope it (transversely) to a contact embedding which is compatible with some supporting open book for the embedded contact 3--manifold and the standard open book which is a supporting open book for $(S^5,\xi_{std})$. 
\end{remark}

We are now ready to prove Theorem~\ref{embeds3s} that says all contact structures on $S^3$ can be embedded into $(S^5,\xi_{std})$ in the isotopy class of the standard embedding and in infinitely many other isotopy classes.
\begin{proof}[Proof of Theorem~\ref{embeds3s}]
For $n\geq0$ let $T_{n}$ be the transverse unknot in $(S^3,\xi_{std})$ with self-linking number $-1-2n$. Recall from \cite{Eliashberg93} there is a unique such transverse knot and as $n$ ranges over the positive integers this is a complete list of transverse unknots in the standard contact structure on $S^3$. Moreover $T_{n+1}$ is the stabilization of $T_n$.  It is easy to check that the contact structure on $S^3$ obtained from the 2--fold cover of $(S^3,\xi_{std})$ branched along $T_{0}$ is $\xi_{std}$. Thus from Proposition~\ref{stabilize} we see that the overtwisted contact structure $\xi_n$ on $S^3$, for $n>0$, is obtained as the 2--fold cover of $(S^3,\xi_{std})$ branched along $T_n$. 

The standardly embedded $(S^3,\xi_{std})$ in $(S^5,\xi_{std})$ has a neighborhood $S^3\times D^2$ contactomorphic to $(S^3\times D^2,\ker (\alpha_{std}+r^2\, d\theta))$, where $\xi_{std}=\ker \alpha_{std}$. Lemma~\ref{isotop} below shows how to create a braided embedding about the standard embedding, whose branched covering map is a 2-fold cyclic covering branched over the unknot, that is smoothly isotopic to the standard embedding. Since any oriented knot is isotopic to a positive transverse knot, Theorem~\ref{MainContactBraid} gives a contact embedding $(S^3,\xi_n), n\geq 1,$ into $(S^5,\xi_{std})$ in this isotopy class of embedding. (We note that once $(S^3,\xi_1)$ is embedded, the argument below will embed the other $(S^3,\xi_n)$, for all $n$, but it is interesting to note that the $\xi_n$, for $n>0$, can all be embedded using a braided embeddings about $(S^3,\xi_{std})$.)

Arguing similarly if we can show that a $(S^3,\xi_n)$ for any $n$ is a 2--fold cover of $(S^3,\xi_1)$ branched along some transverse unknot, then we will have contact embeddings of these contact manifolds into $(S^5,\xi_{std})$. To this end recall \cite{Dymara01, Etnyre13} that in $\xi_1$ there are transverse knots $T'_n$ with self-linking number $-1-2n$ for all $n\in \Z$ whose complements are overtwisted and $T'_{n+1}$ is a stabilization of $T'_n$. Because the complements are overtwisted it is clear that all the 2--fold cyclic covers of $(S^3,\xi_1)$ branched along $T'_n$ are overtwisted contact structures on $S^3$, which we will denote $\eta_n$. Proposition~\ref{stabilize} tells us that $d_3(\eta_n)=d_3(\eta_0)+n$. And so the $\eta_n$ realize all homotopy classes of plane field, and hence by Theorem~\ref{otclass} all overtwisted contact structures, on $S^3$. 

Now consider the braided embeddings from Example~\ref{differentembeddings}. Recall there are infinitely many distinct isotopy classes of embeddings and the embeddings respect the standard open books on $S^3$ and $S^5$ (that is, they send pages to pages and binding to binding). So the induced contact structures on $S^3$ from these embeddings are all supported by the standard open book and hence are all $\xi_{std}$. Now connect summing with the embeddings constructed above give embeddings of all contact structures on $S^3$ into these isotopy classes of smooth embeddings. 
\end{proof}

\begin{lem}\label{isotop}
If $p\colon S^3\to S^3$ is the $k$-fold cyclic branched covering map with branch locus the unknot $U$ then there is a map  $h\colon S^3\to \C$ such that 
\[
e\colon S^3\to S^3\times \C: x\mapsto (p(x),h(x))
\] 
is a braided embedding 
for which the embedding of $S^3\to S^5$ coming from $e$ is isotopic to the standard embedding. 
\end{lem}
\begin{remark}
The braided embedding in Lemma~\ref{isotop}  when $k=2$ can be thought of as a simple case of a ``braid stabilization" in dimension 5.  The notion of such a stabilization has previously been announced by Mori. 
\end{remark}
\begin{proof}
Below we will construct the braided embedding using a specific choice for the branched covering map $p$ but since we claim the result is true for any such map we begin by observing that the braided embedding, up to isotopy, does not depend on the exact choice of $k$-fold branched cover. Specifically we show below that the braided embedding is unchanged, up to isotopy, if the branched locus is changed by a smooth isotopy. We then show that in our situation if $p'\colon S^3\to S^3$ is another $k$-fold cyclic branched covering map with the same branch locus as $p$ then there is an isotopy of $p$ to $p'$ and $h$ to $h'\colon S^3\to D^2$ giving an isotopy of braided embeddings. 

Addressing the first point, suppose we are given a braided embedding $e\colon M\to Y\times \C: x\mapsto (p(x),f(x))$, notice that if the branch locus $B$ is changed by an isotopy then there is an ambient isotopy of $Y$ that induces this isotopy and composing with $p$ gives a family of functions $p_t\colon M\to Y$ that will induce an isotopy of the embedding $M\to Y\times \C$. A similar argument allows us to isotope the branched locus in $M$ (though here we will also need to compose $h$ with the ambient isotopy to maintain a braided embedding). 

For the second point, suppose $p'$ and $p$ are two $k$-fold branched covering maps with branch locus the unknot $U$. Since any orientation preserving isotopy of $S^1$ is isotopic to the identity, we can isotopy $p$, through branched covering maps with branch locus $U$, to a map that agrees with $p'$ on $U$ and since the ramification data for $p$ and $p'$ is the same, by a further isotopy we can assume $p$ agrees with $p'$ in a neighborhood of $U$.  
Now consider $p$ and $p'$ on the complement of the branch loci where they are simply $k$-fold covering maps of the (open) solid torus $S^1\times \R^2$. Using the lifting criteria for covering maps we know there is a diffeomorphism $f\colon S^1\times \R^2\to S^1\times \R^2$ such that $p=p'\circ f$. Moreover by our prior isotopies we know $f$ is the identity map outside some compact set and hence is the identity on all of $S^1\times \R^2$. We can extend $f$ to a diffeomorphism of all of $S^3$ such that $p=p'\circ f$ and clearly this diffeomorphism is isotopic to the identity leaving the branch locus fixed. Thus we have constructed our isotopy of braided embeddings from $(p,h)$ to $(p',h')$. 

\begin{remark}
The braided embedding could depend on $h$. It would be interesting to find explicit non-isotopic braided embeddings realizing a fixed branched cover $p\colon M\to Y$. Is this possible when considering $k$-fold cyclic covers? It certainly is in dimension 3. What about higher dimensions?
\end{remark}

We are thus left to check the lemma is true for a specific choice of unknot and a specific choice of  $h\colon S^3\to \C$. To this end we consider $S^5_\epsilon=\{|z_1|^2+|z_2|^2+|z_3|^2=\epsilon^2\}$ for $\epsilon>0$ in $\C^3$ with coordinates $(z_1,z_2, z_3)$. We then consider the standard embedding of $S^3$ in $S^5$ to be given by $S^3_\epsilon=\{z_3=0\}\cap S^5_\epsilon$ and the unknot in $S^3_\epsilon$ as being given by $U=\{z_2=z_3=0\}\cap S^5_\epsilon$. Denote by  $U'=\{z_1=z_2=0\}$ the $S^1$ in $S^5_\epsilon$ that is complementary to $S^3_\epsilon$ (that is one can see $S^5_\epsilon$ as the join of $S^3_\epsilon$ and $U'$). Notice that $C=S^5_\epsilon-U'$ is diffeomorphic to $S^3_\epsilon\times \C$ by the diffeomorphism
\[
S^3_\epsilon \times \C \to C:((z_1,z_2), z_3)\mapsto \left(\frac{\epsilon z_1}{\sqrt{\epsilon^2+|z_3|^2}}, \frac{\epsilon z_2}{\sqrt{\epsilon^2+|z_3|^2}}, \frac{\epsilon z_3}{\sqrt{\epsilon^2+|z_3|^2}}\right).
\]
and the map 
\[
\pi\colon C\to S^3_\epsilon:(z_1,z_2,z_3)\mapsto \frac{\epsilon}{\sqrt{|z_1|^2+|z_2|^2}} (z_1,z_2)
\] 
is simply the projection map to $S^3$. 

Consider the complex polynomial $p_t(z_1,z_2,z_3)=z_2-t z_3^k$, where $t\in [0,1]$. 
Notice that for a sufficiently small fixed $\epsilon > 0$ the zero set intersected with $S^5_\epsilon$, which we denote by $S_t$, is a transversely cut out sphere in $S^5_\epsilon$ for all $t\in [0,1]$ by \cite[Lemma~2.12]{Milnor68}. 
Consider the map $p\colon S_1\to S^3_\epsilon$ obtained by restricting $\pi$ to $S_1$. We claim this is a $k$-fold covering map branched along $U$. To see this we first note that for each point $(z_1,z_2)\in S^3-U$ we have $z_2\not=0$ so there are precisely $k$ roots of $z_2$ and denoting one of these sheets by $\sqrt[k]{z_2}$ we see that the map $(z_1,z_2)\mapsto \frac{\epsilon}{\sqrt{|z_1|^2+|z_2|^2+|z_2|^{2/k}}}(z_1,z_2, \sqrt[k]{z_2})$ is a local section of  $p\colon S_1\to S^3_\epsilon$. Thus we see that $p$ is a $k$--fold covering map from $S_1-p^{-1}(U)$ to $S^3_\epsilon-U$. Moreover for any $(z_1,z_2)\in U$ we see that $z_2=0$ so there is a unique $k$-th root and the only point in $S_1$ lying above it is $(z_1,0,0)$. 

Thus we see that $S_1$ is a sphere that is braided about the standardly embedded $S^3$ in $S^5$ and realizing a $k$--fold cyclic branched cover over the unknot $U$. The spheres $S_t$ for $t\in[0,1]$ provide an isotopy from our braided sphere $S_1$ to the sphere $S_0=\{z_2=0\}$ which is clearly isotopic to the standardly embedded sphere $S^3_\epsilon$.
\end{proof}

We now turn to the proof of Theorem~\ref{allhaveembed} concerning the embeddings of overtwisted contact structures on 3--manifolds $M$ with no 2--torsion in their second cohomology.
\begin{proof}[Proof of Theorem~\ref{allhaveembed}]
The vanishing of the first Chern class is a necessary condition for contact embeddings of contact $3$-manifolds into $(S^5,\xi_{std})$ by 
Theorem~\ref{thm:obstruct}. From Proposition~\ref{oneembed} we know that every 3-manifold $M$ has some contact structure $\xi$ that embeds in $(S^5, \xi_{std})$. Now using Lemma~\ref{cconnectsum} we know that $\xi\# \xi_n$ embeds for all overtwisted contact structures $\xi_n$ on $S^3$. Using Proposition~\ref{gompfclass} and Theorem~\ref{otclass} we see from the fact that there is no 2--torsion in the second cohomology of $M$ that every overtwisted contact structure with trivial first Chern class on $M$ is of the form $\xi\#\xi_n$ for some $n$ and thus they all embed. 
\end{proof}

We now consider embedding tight contact structures on lens spaces into $(S^5,\xi_{std})$.
\begin{lem}\label{alltight}
A tight contact structure $\xi$ on a lens space $L(p,q)$ contact embeds in $(S^5,\xi_{std})$ if and only if $c_1(\xi)=0$. 
\end{lem}
\begin{proof}
We begin by recalling the classification of tight contact structures on $L(p,q)$. Given $p>q>1$ consider the continued fraction expansion of $-p/q$:
\[
-p/q=a_1-\frac{1}{a_2-\frac{1}{\ldots - \frac{1}{a_n}}},
\]
where each $a_i\leq -2$. It is well known that $L(p,q)$ is obtained from surgery on the link on the left in Figure~\ref{fig:lpq}. 
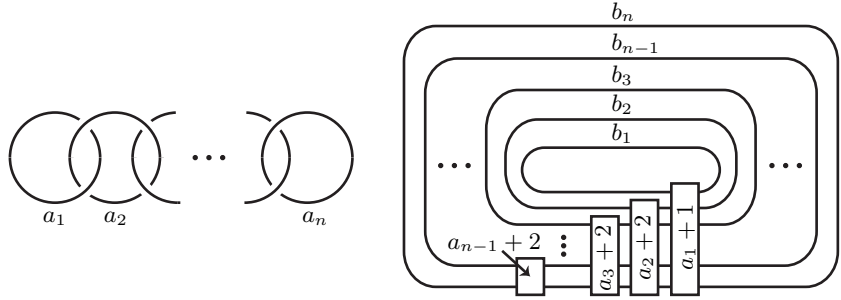
\begin{figure}[htb]{\small
\begin{overpic}
{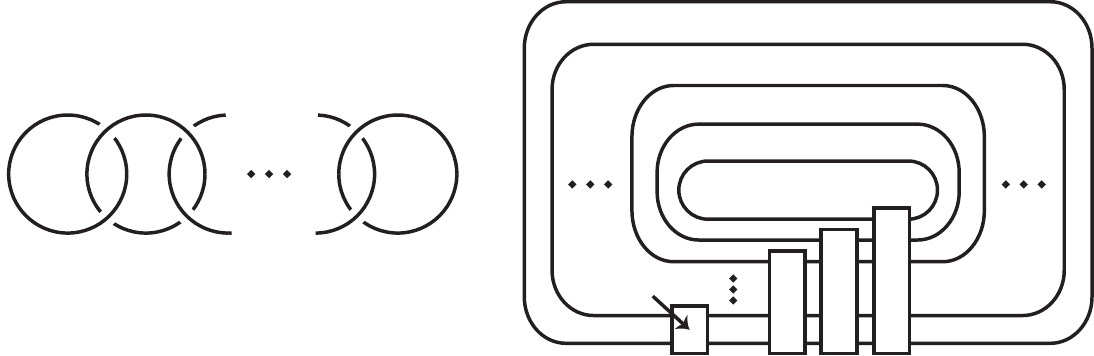}
\put(255,10){\rotatebox{90}{$a_1+1$}}
\put(239,6){\rotatebox{90}{$a_2+2$}}
\put(224,2){\rotatebox{90}{$a_3+2$}}
\put(168,18){$a_{n-1}+2$}
\put(230,59){$b_1$}
\put(230,70){$b_2$}
\put(230,81){$b_3$}
\put(230,93){$b_{n-1}$}
\put(230,105){$b_{n}$}
\put(15,28){$a_1$}
\put(38,28){$a_2$}
\put(113,28){$a_n$}
\end{overpic}}
\caption{On the left is a surgery picture for $L(p,q)$ in terms of a continued fractions expansion for $-p/q$. On the right is another surgery picture for $L(p,q)$.}
\label{fig:lpq}
\end{figure}
Honda \cite{Honda00a} and Giroux \cite{Giroux00} proved that there is a one-to-one correspondence between tight contact structures on $L(p,q)$ and the contact structures obtained from Legendrian surgery on all possible Legendrian realizations of the link on the left hand side of Figure~\ref{fig:lpq}. 

We claim that on each $L(p,q)$ there is exactly 0 or 1 tight contact structure with $c_1=0$, and if it exists, it comes from Legendrian surgery on a Legendrian realization of the link in Figure~\ref{fig:lpq} with all components having rotation number 0. This will follow if we can see that the only contact structure obtained from Legendrian surgery with $c_1=0$ is the one on a link that has all rotation numbers 0. This statement is almost contained in  \cite[Proposition~1.7]{Giroux00}, but to establish it we argue as follows. 

Arguing by contradiction we assume that there is a Legendrian realization of the link in Figure~\ref{fig:lpq} with some rotation numbers non-zero and the contact structure $\xi$ on $L(p,q)$ obtained by surgery on this link has $c_1(\xi)=0$.  By Theorem~\ref{buildstein} we have a Stein domain $X$ with boundary $L(p,q)$ and inducing the contact structure $\xi$. Moreover $c_1(X)=\sum_{i=1}^k r(L_i) h_i$ which is non-zero. Now  \cite[Corollary~4.10]{Gompf98} says that an oriented plane field is homotopic to itself with reversed orientation if and only if its first Chern class is 0. So $\xi$ is homotopic, as a plane field, to $-\xi$. Notice if $J$ is the complex structure on $X$ then $\xi$ is the set of $J$-complex tangencies to $\partial X$ and $-\xi$ is the set of $\overline{J}$-complex tangencies (where $\overline{J}$ is the conjugate complex structure on $X$). Notice that $c_1(\overline{J})=- c_1(J)\not = c_1(J)$, since the cohomology of $X$ is free and $c_1(J)\not=0$, and thus $\xi$ and $-\xi$ are not isotopic as contact structures due to a result of Lisca and Matic \cite[Theorem1.2]{LiscaMatic97}. So we have found two contact structures in the same homotopy class of plane field, but this contradicts \cite[Theorem~1.1]{Giroux00} and \cite[Proposition~4.24]{Honda00a} which says that the tight contact structures on lens spaces are all in distinct homotopy classes of plane fields. Thus our assumption must have been false. 

We now note that the surgery diagram on the left of Figure~\ref{fig:lpq} can be transformed by simple handle slides to the ``rolled up" diagram on the right. In the figure the surgery coefficients are
\[
b_k=2(k-1)+ \sum_{i=1}^k a_i.
\]
Notice that the surgery coefficients are decreasing moving from the inside circle out. One may choose a Legendrian realization of the innermost circle with $tb=b_1+1$, then take a push-off of it and stabilize it enough times to get a Legendrian with $tb=b_2+1$ and continue until we have a Legendrian link on which Legendrian surgery will yield $L(p,q)$. One may check that all the tight contact structures on $L(p,q)$ may be obtained this way (see for example \cite{EtnyreOzbagci06}). Thus the contact structures with $c_1=0$ exist only on lens spaces where all the $a_i$ are even and a Legendrian surgery picture of them only have $r=0$ Legendrian unknots and so is of the form shown in Figure~\ref{fig:legsurg}. 
\begin{figure}[htb]{\small
\begin{overpic}
{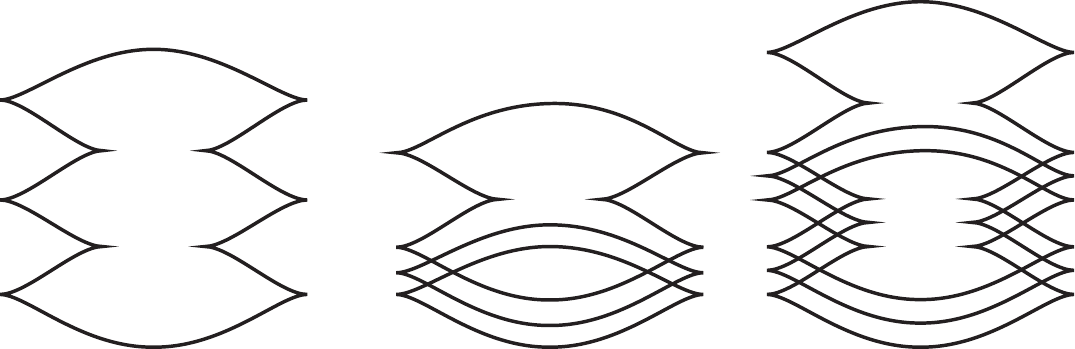}
\end{overpic}}
\caption{Legendrian surgery diagrams for the $c_1=0$ tight contact structures on $L(6,1)$, $L(10,7)$, and $L(24,7)$.}
\label{fig:legsurg}
\end{figure}

We indicate how to put these Legendrian knots on the page of an open book supporting the standard tight contact structure on $S^3$, for more details see \cite{Etnyre04b, Etnyre06}. Figure~\ref{planarforr0} shows a planar surface $\Sigma$ with 8 boundary components.
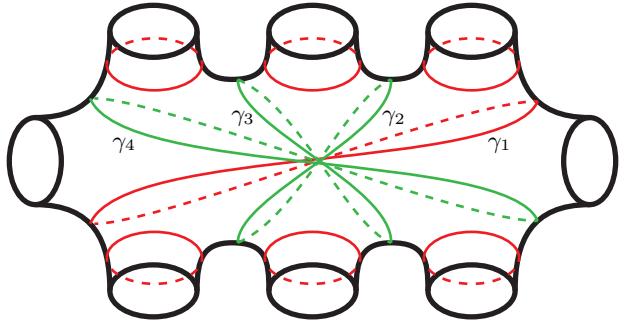
\begin{figure}[htb]{\small
\begin{overpic}
{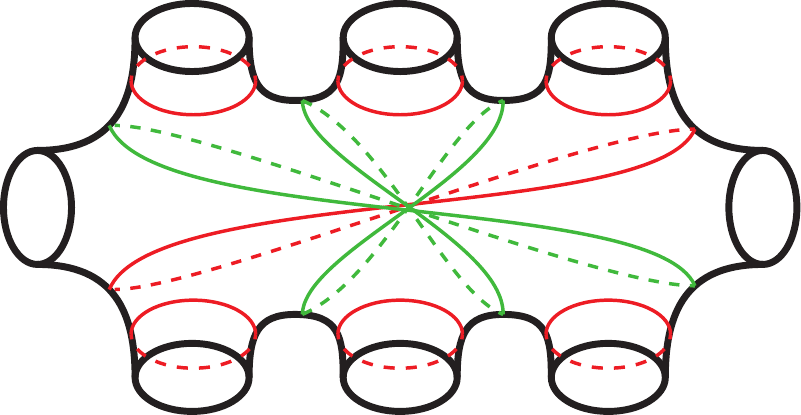}
\put(40,65){$\gamma_4$}
\put(85,75){$\gamma_3$}
\put(142,75){$\gamma_2$}
\put(182,65){$\gamma_1$}
\end{overpic}}
\caption{Open book for $S^3$ and Legendrian unknots with rotation 0.}
\label{planarforr0}
\end{figure}
If $\phi$ is the composition of positive Dehn twists about the red curves then $(\Sigma,\phi)$ supports $\xi_{std}$ on $S^3$ (note it is clear that this open book is a stabilization of the annular open book for $S^3$). Using Giroux's Legendrian realization principle the curves $\gamma_i$ can each be realized by Legendrian knots on the page of the open book. (Since pages have transverse boundary, this is not a standard application of the realization principle, but can nonetheless be done, see the proof of Theorem~5.11 in \cite{Etnyre06}.) Notice that $(\Sigma,\phi)$ is obtained from the annular open book supporting $(S^3,\xi_{std})$ (which we think of as a neighborhood of $\gamma_1$) by stabilizing the open book three times on each boundary component of the annulus. Now $\gamma_2$ is obtained from $\gamma_1$ by sliding over one of the open book stabilizations applied to each boundary component of the annulus. It is known, see \cite[Lemma~3.3]{Etnyre04b}, that this implies that the Legendrian realization of $\gamma_2$ is obtained from a copy of the Legendrian realization of $\gamma_1$ by a positive and a negative stabilization. Similarly, $\gamma_i$ is a push-off of $\gamma_{i-1}$ followed by a positive and a negative stabilization. Thus since the Legendrian realization of the core curve in the annular open book supporting $(S^3,\xi_{std})$ represents the Legendrian unknot with $tb=-1$ and $r=0$, it is clear that the $\gamma_i$ realize Legendrian unknots with $tb=-2i+1$ and $r=0$.  For any positive integer $k$ there are clearly analogous pictures on which we can realize all Legendrian unknots with rotation 0 and Thurston-Bennequin invariant odd integers between $-1$ and $-2k+1$. Performing Legendrian surgeries on the $\gamma_i$ is equivalent to adding right handed Dehn twists to the monodromy along the corresponding curve. Thus it is clear all the tight contact structures on lens spaces with $c_1=0$ can be realized by open books analogous to the one shown in Figure~\ref{planarforr0}.

We can stabilize the open book $(\Sigma,\phi)$ to get the open book shown in Figure~\ref{genusforr0} on which we still see the Legendrian knots $\gamma_i$ and can still add Dehn twists to them in order to realize all tight contact structures on $L(p,q)$ with trivial first Chern class. We call this surface $\Sigma'$ and the new monodromy, which is a composition of right handed Dehn twist about the red curves from Figure~\ref{genusforr0}, $\phi'$. 
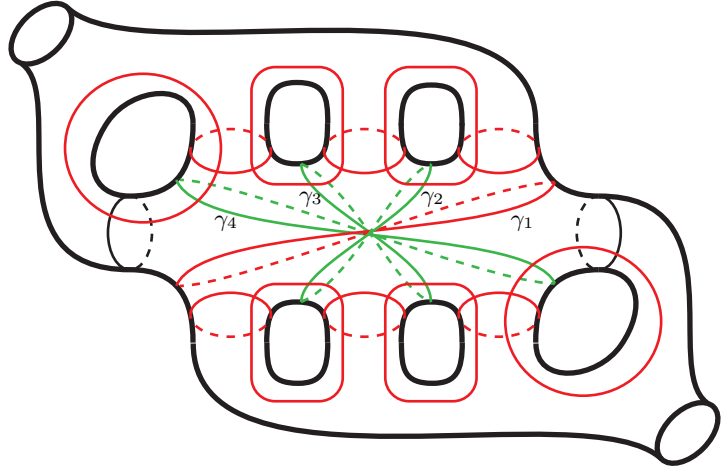
\begin{figure}[htb]{\small
\begin{overpic}
{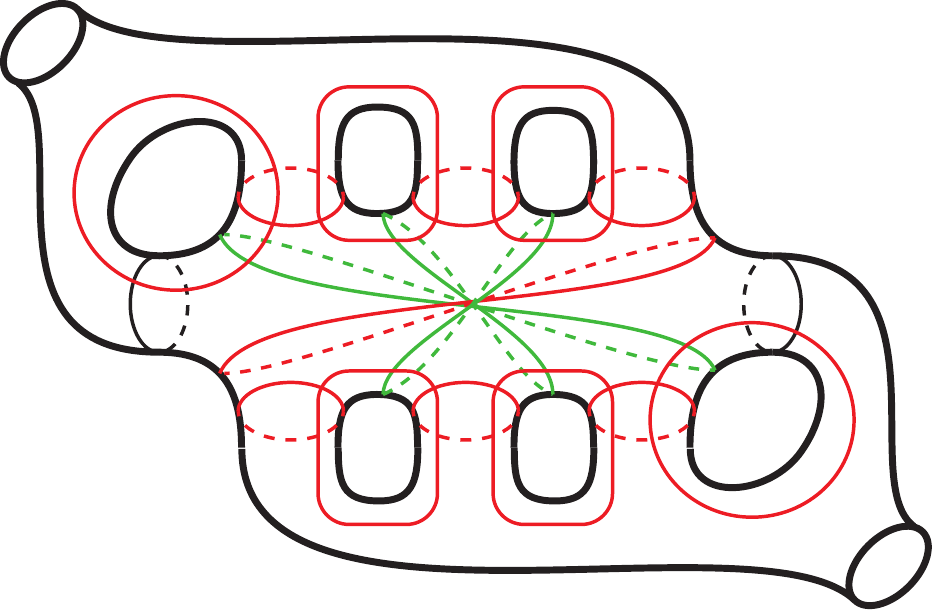}
\put(190,91){$\gamma_1$}
\put(156,100){$\gamma_2$}
\put(110,100){$\gamma_3$}
\put(78,91){$\gamma_4$}
\end{overpic}}
\caption{Stabilization of $(\Sigma,\phi)$.}
\label{genusforr0}
\end{figure}
The top picture in Figure~\ref{bcoverforr0} is a symmetric version of $\Sigma'$ from Figure~\ref{genusforr0} (notice that we have put a half twist about the ``waist" of the surface). 
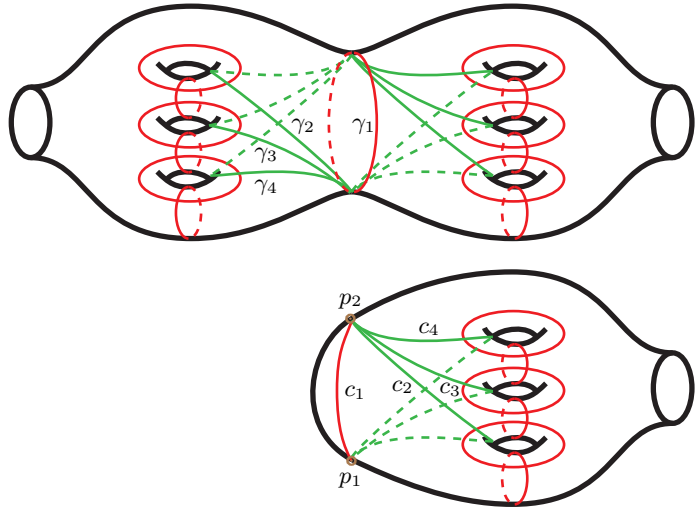
\begin{figure}[htb]{\small
\begin{overpic}
{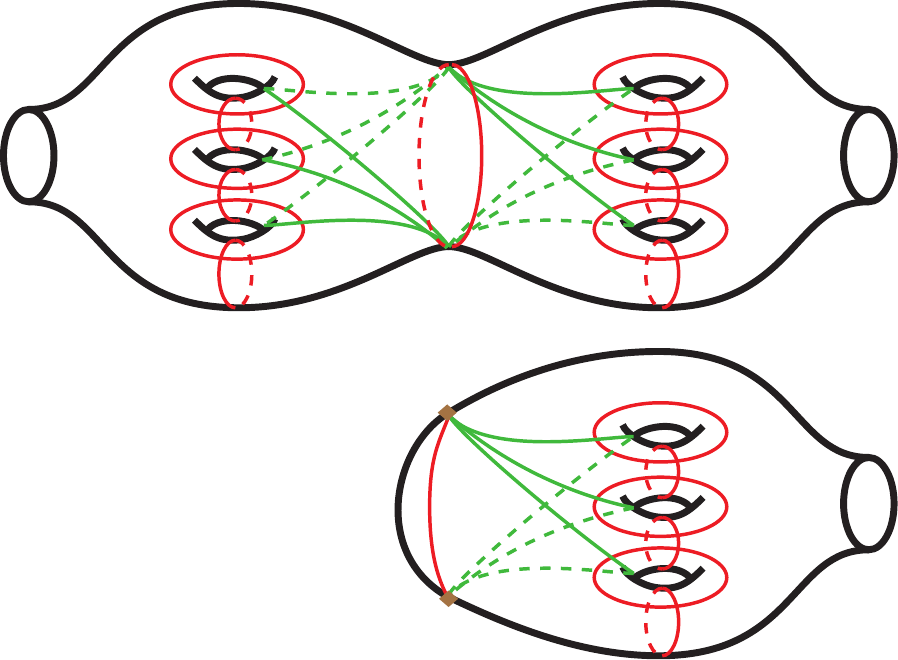}
\put(93,120){$\gamma_4$}
\put(93,132){$\gamma_3$}
\put(107,143){$\gamma_2$}
\put(130,143){$\gamma_1$}
\put(155,66){$c_4$}
\put(163,43){$c_3$}
\put(145,44){$c_2$}
\put(127,42){$c_1$}
\put(125,77){$p_2$}
\put(125,9){$p_1$}
\end{overpic}}
\caption{On the top is an open book for $S^3$ on which one can see Legendrian unknots with rotation 0. On the bottom is the quotient of the top surface by the obvious involution.}
\label{bcoverforr0}
\end{figure}
There is an involution of the surface on the top of the figure given by rotation by $\pi$ around a vertical line piercing the center of the surface. Quotienting by this action yields the surface $F$ shown on the bottom of Figure~\ref{bcoverforr0}. Clearly $\Sigma'$ is the 2-fold branched cover of $F$ branched along the two brown points $\{p_1,p_2\}$ shown in the figure. Let $\psi$ be the composition of a right handed Dehn twist about each red simple closed curve in the bottom picture in Figure~\ref{bcoverforr0} (compose the Dehn twists starting from the bottom curve and working up the chain). The open book $(F,\psi)$ is obtained from the open book $(D^2,id_{D^2})$ by a sequence of stabilizations and hence supports the standard tight contact structure on $S^3$. The two branch points  trace out a two component transverse link  $T$ in the open book (that is $\{p_1,p_2\}\times [0,1]/\sim$ in the mapping torus part of the open book). The two fold branched cover of $(S^3,\xi_{std})$ over this link will result in the contact structure supported by the open book $(\Sigma,\phi'')$ where $\phi''$ is $\phi'$ except the Dehn twist about $\gamma_1$ is not used. If the link $T$ is changed by adding a half twist about one of the arcs $c_i$ in $F$ shown in Figure~\ref{bcoverforr0}, then the monodromy of the branched cover changes by adding a right handed Dehn twist about the corresponding $\gamma_i$. Thus we see there is a transverse link in $(S^3,\xi_{std})$ for which we can take the two fold cover of $S^3$ branched along this link to obtain any tight contact structure on a lens space with $c_1=0$. Now by Corollary~\ref{contctforcyclic} we see that all these tight contact structures contact embed in $(S^5,\xi_{std})$.
\end{proof}

We end the paper with a proof of Theorem~\ref{embedall} that establishes the embeddability into $(S^5,\xi_{std})$ of contact structures with vanishing first Chern class on $S^1\times S^2$, $T^3$, and some lens spaces.  
\begin{proof}[Proof of Theorem~\ref{embedall}]
There is a unique tight contact structure $\xi_t$ on $S^1\times S^2$ that is supported by the open book with annulus page and identity monodromy. Thus it is easy to see it is obtained as the double cover of $(S^3,\xi_{std})$ branched along the two component unlink with both components being transverse knots of self-linking $-1$. Now Corollary~\ref{contctforcyclic} allows us to embed $\xi_t$ into $(S^5,\xi_{std})$. Since there is no 2--torsion in the homology of $S^1\times S^2$ we see from Theorem~\ref{allhaveembed} that all overtwisted contact structures with $c_1=0$ also embed. 

Similarly for $T^3$ we see that all overtwisted contact structures with $c_1=0$ embed in $(S^5,\xi_{std})$. A complete list of tight contact structures on $T^3$ is given by
\[
\xi^{T^3}_n=\ker (\cos 2\pi n z\, dx + \sin 2\pi n z\, dy),
\]
where $T^3$ is thought of as $[0,1]^3$ with opposite sides identified by translation and $n$ is a positive integer, see \cite{Kanda97}.  One my check that $c_1(\xi^{T^3}_n)=0$ for all $n$.  It is easy to see that $\xi^{T^3}_n$ is an $n$--fold (ordinary) cyclic cover of $\xi^{T^3}_1$ where it is the $z$-coordinate that is unwrapped $n$ times. In addition, one may check that $\xi^{T^3}_1$ is the contact structure induced on the boundary of the unit cotangent bundle $T^*T^2$ by the Liouville form (or consult \cite{Kanda97}). We notice that if $h\colon T^3\to S^1$ is projection onto the $z$-coordinate thought of as the unit circle in $\C$ then the proof of Theorem~\ref{cyclicbraid} gives a braided embedding of the $n$--fold (ordinary) cover of $T^3$ into $T^3\times D^2$ and since there is no branch locus to worry about Theorem~\ref{MainContactBraid} clearly gives a contact embedding of $(T^3,\xi^{T^3}_n)$ into $(T^3\times D^2, \ker (\alpha_1+r^2\, d\theta))$, where $\alpha_1$ is the contact form for $\xi^{T^3}_1$. Thus if we can embed $(T^3,\xi^{T^3}_1)$ into $(S^5,\xi_{std})$ then we will have an embedding of all tight contact structures on $T^3$. 

Recall there are many embeddings of a Legendrian $T^2$ into $(S^5,\xi_{std})$. They can be constructed in various ways, for example using front projections, see \cite{EkholmEtnyreSullivan05a}. 
By the neighborhood theorem for Legendrian submanifolds, see \cite[Theorem~2.5.8]{Geiges08}, a Legendrian $T^2$ has a neighborhood contactomorphic to a neighborhood of the zero section in the 1-jet space $T^*T^2\times \R$ with the contact structure $\ker(dz-\lambda)$, where $\lambda$ is the Liouville 1--form on $T^*T^2$ and $z$ is the coordinate on $\R$. Let $S_\epsilon$ be the $\epsilon$-sphere bundle in $T^*T^2$. As mentioned above $\lambda$ restricted to $S_\epsilon$ is a contact 1--form defining $\xi_1^{T^3}$ and thus $(T^3,\xi_1^{T^3})$ contact embeds in $(S^5,\xi_{std})$. 
\begin{remark}
Once can also embed $(T^3,\xi^{T^3}_1)$ into $(S^5,\xi_{std})$ explicitly as 
\[
\{(z_0,z_1,z_2): |z_0|^2+|z_1|^2+|z_2|^2=r, z_0z_1z_2=1\}
\]
for sufficiently large $r>0$. To see that this is indeed the claimed contact manifold one notes that this set is the boundary of the Stein domain $\{(z_0,z_1,z_2): |z_0|^2+|z_1|^2+|z_2|^2\leq r, z_0z_1z_2=1\}$ which is diffeomorphic to  $D^2\times T^2$. 
\end{remark}

Turning to contact structures on lens spaces $L(p,q)$, the theorem follows when $p$ is odd from Theorem~\ref{allhaveembed} and Lemma~\ref{alltight}. We cannot use Theorem~\ref{allhaveembed} to embed all the overtwisted contact structures on $L(p,q)$ when $p$ is even since the 2-torsion in the first homology group means that $c_1$ and $d_3$ do not determine the isotopy class of an overtwisted contact structure, see Section~\ref{htpyclasses}.

We now consider the case of contact structures on $L(p,q)$ when $p$ is even and $q=1$. From Lemma~\ref{alltight} we know the tight contact structure on $L(p,1)$ with $c_1=0$ embeds and by connect summing with the overtwisted contact structures on $S^3$ we see that all overtwisted contact structures with the same $\Gamma$ invariant (see Section~\ref{htpyclasses}) will also embed. So we are left to see that we can embed one overtwisted contact structure with $c_1=0$ and different $\Gamma$ invariant (recall there are only two possible $\Gamma$ invariants on $L(p,q)$ for a given Chern class). Then by connect summing with the overtwisted contact structures on $S^3$ we will have embedded all contact structures on $L(p,1)$ with $c_1=0$. 

To this end consider the surgery pictures for $L(p,1)$ given in Figure~\ref{lp1}. 
\begin{figure}[htb]{\small
\begin{overpic}
{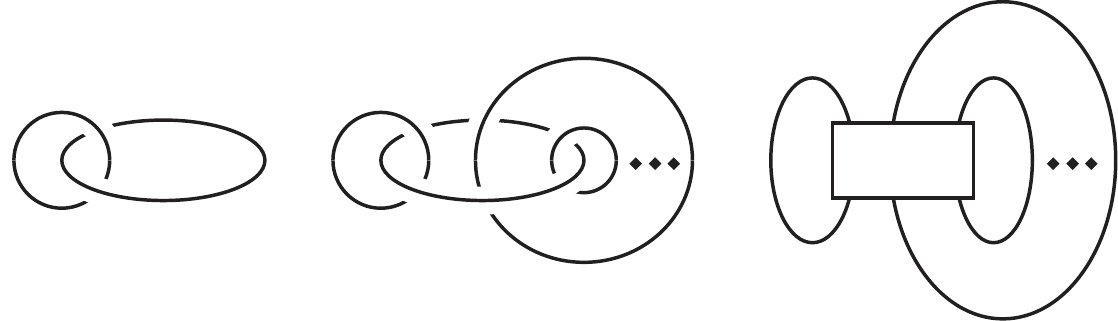}
\put(61,29){$-p$}
\put(-5,60){$\gamma,  \mathbf{f}$}
\put(123,27){$1$}
\put(180,55){$1$}
\put(190,75){$1$}
\put(85,60){$\gamma,  \mathbf{f}$}
\put(297,66){$0$}
\put(315,85){$0$}
\put(254,43){$-1$}
\put(220,75){$\gamma,  \mathbf{f}-1$}
\end{overpic}}
\caption{On the left is the standard surgery picture for $L(p,1)$ with the generator for homology $\gamma$ and framing $\mathbf{f}$ on $\gamma$ shown. In the middle the same figure after blowing up $p+1$ curves. On the right the same figure after blowing down the horizontal $1$-framed unknot.}
\label{lp1}
\end{figure}
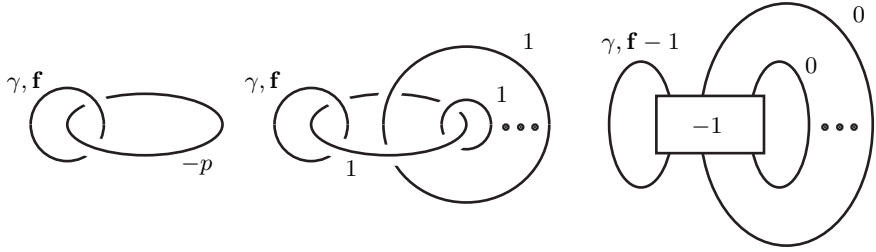
The left hand surgery diagram can be written as Legendrian surgery on a Legendrian unknot of the type shown on the left in Figure~\ref{fig:legsurg}. We have seen that we can embed this tight contact structure $\xi$ in $(S^5,\xi_{std})$. The surgery diagram describes a 4--manifold with a unique spin structure on it and this spin structure induces the spin structure $\mathbf{s}$ on $L(p,1)$ that extends over a 2--handle attached to $\gamma$ with even framing. Thus it corresponds to the empty characteristic sub-link $L'$. From this Equation~\eqref{gammainvt} tells us that $\Gamma_{\xi}(\mathbf{s})=0$. 

Now consider the surgery picture on the right in Figure~\ref{lp1}. This can be realized as $(+1)$-contact surgery on $p+1$ copies of the Legendrian unknot with $tb=-1$ and gives a contact structure $\xi'$ on $L(p,1)$ that can be embedded in $(S^5,\xi_{std})$. (This is clear since $\xi'$ is supported by the open book with annular page and monodromy the $p^\text{th}$ power of the left handed Dehn twist about the core of the annulus. And this open book can clearly be realized as a 2-fold branched cyclic cover.) The spin structure $\mathbf{s'}$ corresponding to the empty characteristic sub-link of this surgery diagram will extend over a 2--handle attached to $\gamma$ with even framing. If Equation~\eqref{gammainvt} held for general contact surgeries then we could conclude that $\Gamma_{\xi'}(\mathbf{s'})=0$. We believe that Equation~\eqref{gammainvt} does indeed hold for general contact surgeries, but as a proof does not exist in the literature we provide a different argument for the computation of $\Gamma_{\xi'}$ below, but first notice that the computation shows $\xi'$ and $\xi$ have different $\Gamma$ invariant. By tracking the framings on $\gamma$ through the surgery pictures in Figure~\ref{lp1} we see that $\mathbf{s}$ and $\mathbf{s'}$ are distinct spin structures on $L(p,1)$ and thus $\Gamma_{\xi'}(\mathbf{s})=\frac{p}{2} [\gamma]$ where $[\gamma]$ is the homology class of $\gamma$. In particular $\xi'$ has different $\Gamma$ invariant than the tight contact structure $\xi$. 

We now rigorously establish that $\Gamma_{\xi'}(\mathbf{s'})=0$. To this end recall that for any open book $(\Sigma, \phi)$ there is an orientation reversing diffeomorphism $\Psi\colon M_{(\Sigma,\phi)}\to M_{(\Sigma,\phi^{-1})}$ given on the mapping cylinder $T_\phi$ by $(p,t)\mapsto (p,1-t)$ and extended to the neighborhoods of the binding in the obvious way. For simplicity we now homotope the contact structure $\xi_{(\Sigma,\phi)}$ to the plane field $\widehat\xi_\phi$ that is given by the tangents to the pages on $T_\phi$ and by the usual formula on the neighborhoods of the binding. (More specifically, each binding component has a neighborhood $S^1\times D^2$ with coordinates $(\phi,(r,\theta))$, where $D^2$ is a disk of radius $\epsilon$, and the plane field will be given by the kernel of $g(r)\, d\phi+f(r)\, d\theta$ for functions $g$ and $f$ that are equal to $1$ and $r^2$, respectively, near $r=0$, $0$ and $1$, respectively, near $r=\epsilon$, and satisfy $f'g-g'f\geq 0$.) We similarly have $\widehat\xi_{\phi^{-1}}$. One may easily see that as oriented plane fields $\Psi_*(\widehat{\xi}_\phi)$ and $\widehat\xi_{\phi^{-1}}$ agree outside a neighborhood of the binding and differ from one another by a half-Lutz twist along the binding. Since the binding is null-homologous the 2-dimensional difference class between  $\Psi_*(\widehat{\xi}_\phi)$ and $\widehat\xi_{\phi^{-1}}$  is 0, in other words they are homotopic over the 2-skeleton and thus have the same $\Gamma$ invariant. Returning to our situation let $(\Sigma,\phi)$ be the open book for $\xi'$ described above. Now $(\Sigma,\phi^{-1})$ supports the contact structure obtained from Legendrian surgery on $p-1$ parallel copies of the maximum Thurston-Bennequin invariant unknot. Thus $\Gamma_{\widehat{\xi}_{\phi^{-1}}}$ can be computed from Equation~\eqref{gammainvt} to be $\Gamma_{\widehat{\xi}_{\phi^{-1}}}(\mathbf{s'})=0$ where $\mathbf s'$ is as above (that is the spin structure that extends over a 2--handle attached along $\gamma$ with even framing). According to \cite[Corollary~4.9]{Gompf98} $\Gamma$ changes sign when the orientation on the ambient manifold is reversed and is preserved under orientation preserving diffeomorphisms. Thus if $\Psi$ is the diffeomorphisms above then $\Gamma_{\xi'}(\Psi^{-1}_*({\mathbf s'}))=\Gamma_{\widehat\xi_\phi}(\Psi^{-1}_*({\mathbf s'}))=-\Gamma_{\widehat\xi_{\phi^{-1}}}(\mathbf{s'})=0$. Moreover notice that $\gamma$ sits on a page of the open book and is taken to a curve on a page of the open book by $\Psi$ and so the framing given by the pages is clearly preserved. Thus $\Psi^{-1}_*({\mathbf s'})$ is the spin structure called $\mathbf s'$ from the previous paragraph and our computation is complete. 

Now turning to $L(p,p-1)$. Notice that $L(p,p-1)$ is simply $L(p,1)$ with the reversed orientation. So taking the open books $(\Sigma, \phi)$ and $(\Sigma', \psi)$ for $\xi$ and $\xi'$, respectively, on $L(p,1)$ above we can consider $(\Sigma, \phi^{-1})$ and $(\Sigma', \psi^{-1})$. These are open books for $L(p,p-1)=-L(p,1)$. Because the contact structures supported by  $(\Sigma, \phi)$ and $(\Sigma', \psi)$ are not homotopic on the 2-skeleton of $L(p,1)$ (since their $\Gamma$ invariants are distinct) neither will the contact structures supported by  $(\Sigma, \phi^{-1})$ and $(\Sigma', \psi^{-1})$ on $L(p,p-1)$. (This should be clear from our discussion in the previous paragraph.) Thus we can embed into $(S^5,\xi_{std})$ two contact structures on $L(p,p-1)$ with different $\Gamma$ invariants and as discussed above this is enough to embed all overtwisted contact structures. So combined with Lemma~\ref{alltight} we have completed the proof of the theorem in the case of $L(p,p-1)$. 
\end{proof}

\def\cprime{$'$} \def\cprime{$'$}

\end{document}